\documentclass{amsart}

\usepackage{amsmath}
\usepackage{amssymb}
\usepackage{tikz}
\usetikzlibrary{matrix,arrows}
\usepackage[utf8]{inputenc}

\usepackage{cite}
\usepackage{tabularx}
\usepackage{url}

\DeclareMathOperator{\Hom}{Hom}
\DeclareMathOperator{\lcm}{lcm}
\DeclareMathOperator{\Ima}{Im}
\DeclareMathOperator{\id}{Id}
\DeclareMathOperator{\PD}{PD}
\DeclareMathOperator{\SW}{SW}
\DeclareMathOperator{\rk}{rank}
\DeclareMathOperator{\SO}{SO}

\newcommand{\IZ}{\mathbb{Z}}
\newcommand{\IR}{\mathbb{R}}
\newcommand{\IQ}{\mathbb{Q}}
\newcommand{\IC}{\mathbb{C}}
\newcommand{\IN}{\mathbb{N}}
\newcommand{\IK}{\mathbb{K}}

\newcommand{\Eilen}{\operatorname{K}}

\newcommand{\Map}[3]{#1\colon #2 \rightarrow #3}

\newcommand{\Thur}[1]{\left\| #1 \right\|_T}
\newcommand{\TorN}[1]{\left\| #1 \right\|_{tor}}
\newcommand{\Comp}[1]{\operatorname{x} #1} 
\newcommand{\TEuler}[1]{\chi_- #1}

\newcommand{\ZkZ}[1]{\IZ_{#1}}

\newcommand{\W}[1]{\operatorname{W}(#1)}

\newcommand{\width}{\operatorname{width}}
\newcommand{\Spinc}{\operatorname{Spin}^c}

\newcommand{\GL}[1]{\operatorname{GL}(#1)}

\newcommand{\Bas}[1]{\operatorname{bas}#1}

\newcommand{\MapT}{\operatorname{M}}

\theoremstyle{plain}
\newtheorem{thm}{Theorem}[section]
\newtheorem*{thm*}{Theorem}
\newtheorem{prop}[thm]{Proposition}
\newtheorem{lem}[thm]{Lemma}

\theoremstyle{definition}
\newtheorem{defn}[thm]{Definition}
\newtheorem{constr}[thm]{Construction}

\theoremstyle{remark}

\newtheorem{rem}[thm]{Remark}

\begin{document}
\title{Minimal genus in circle bundles over 3-manifolds}
\author{Matthias Nagel}
\email{nagel@cirget.ca}
\urladdr{http://thales.math.uqam.ca/~matnagel/}
\address{Département  de  Mathématiques,  Université du Québec à Montréal, Canada}

\begin{abstract}
An estimate is proven for the genus function
in circle bundles over irreducible 3-mani\-folds. 
This estimate is an equality in many cases,
and it relates the minimal genus of the surfaces
representing a given homology class to 
the self-intersection of the class and the Thurston norm 
in the underlying 3-manifold.

The main theorem extends a theorem of Friedl-Vidussi to graph manifolds
and removes their restriction on the Euler class of the circle bundle.
The article also contains computations, of twisted Reidemeister torsion of
graph manifolds, that are of independent interest.
\end{abstract}
\maketitle
\section{Introduction}
Thurston \cite{Thurston86} considered the following question:
What is the minimal complexity of a surface representing
a fixed homology class?
He quantified the complexity by the following adaption
of the Euler characteristic. 
For a closed surface $\Sigma$ with components
$\Sigma_i$ he defined 
\begin{align*}
\TEuler (\Sigma) := \sum_i \max \left(0, - \chi(\Sigma_i)\right),
\end{align*}
where $\chi(\Sigma_i)$ denotes the Euler characteristic of $\Sigma_i$.

In a closed $3$-manifold $N$ one can associate
to every $\sigma \in H_2(N;\IZ)$ a natural number $\Thur \sigma$,
which is the minimum of the numbers $\TEuler(\Sigma)$ taken over all the embedded surfaces $\Sigma$ representing
the class $\sigma$. Thurston proved that this assignment defines a semi-norm $\Thur .$ on $H_2(N;\IZ)$,
i.e. the inequality $\Thur{\sigma + \rho} \leq \Thur \sigma + \Thur \rho$ and the equality
$\Thur {n \sigma} = |n| \Thur \sigma$ hold for all $n \in \IZ$ and all $\sigma, \rho \in H_2(N;\IZ)$.
 
Verbatim the same definition gives rise to a complexity function 
$\Comp: H_2(W; \IZ) \rightarrow \IN$ on a $4$-manifold $W$. 
It is more mysterious and much less understood than the Thurston norm $\Thur .$.
In fact this complexity is not semi-linear as can be easily seen by the adjunction
inequality.

To observe the transition from dimension three to four 
a first step might be to look at circle bundles over $3$-manifolds. 
In this case we obtain the following in Theorem~\ref{Thm:AResult}.
\begin{thm*}
Let $N$ be an irreducible and closed $3$-manifold.
Assume that $N$ is not a Seifert fibred space and not covered by a torus bundle.
Let $\Map p W N$ be an oriented circle bundle. Then each class
$\sigma \in H_2(W;\IZ)$ satisfies the inequality 
\begin{align*}
\Comp(\sigma) \geq | \sigma \cdot \sigma | + \Thur {p_* \sigma}.
\end{align*}
\end{thm*}
The inequality is stronger than the adjunction inequalities on $W$.
It has also the tendency to be sharp. 
See \cite[Section 4]{Friedl14} for a discussion of this, which is
also relevant for graph manifolds.

For the trivial circle bundle such an estimate was proven 
by Kronheimer \cite{Kronheimer99} using a perturbation of
the Seiberg-Witten equation. Later Friedl and Vidussi \cite[Theorem 1.1]{Friedl14} considered
circle bundles over irreducible $3$-mani\-folds with virtually
RFRS fundamental group. They obtained the above estimate for all but finitely many
circle bundles over each such manifold.

A key ingredient in their proof
was that they can approximate non-zero cohomology classes
by fibred classes in covers. 
This approach does not work for graph
manifolds as they might not admit any cover which fibres, see 
\cite[Theorem E]{Neumann97} for an obstruction.
A suitable substitution for this approach is the existence of a finite cover 
in which the torsion norm of Turaev and the Thurston norm agree. 
The torsion norm of Turaev is defined
by twisted Reidemeister torsion and the first 
part of the paper concentrates on calculating the torsion norm 
for some classes of graph manifolds.

The rest of the proof follows roughly along the following lines:
One of the few ways to obtain an estimate as above is to use
the adjunction inequality. See \cite{Lawson97} for a survey on these techniques. 
To use these one needs to know the 
Seiberg-Witten basic classes on the total space of the circle bundle.
By a result of Baldridge \cite[Theorem 1]{Baldridge01},
one can describe the Seiberg-Witten
basic classes of the total space in terms of basic classes
of the base. The basic classes
on a $3$-manifold can be understood
in terms of the maximal abelian torsion 
as described in \cite[Section IX.1.2]{Turaev02}.
By using results relating the torsion norm
with the Thurston norm, we find enough basic classes that 
the adjunction inequality will lead to the estimate above.
Like Friedl and Vidussi we have to deal with possible cancellations
in Baldridge's formula. They had to exclude finitely many
circle bundles over each $3$-manifold. To check whether a given circle bundle
is excluded or not requires knowledge of the Seiberg-Witten basic classes
of the $3$-manifold in question.
We can avoid these cancellations 
by changing the Euler class away from a representing surface of minimal genus.
The necessary freedom to do so is obtained by going up to finite covers.
Therefore we generalise their theorem in two ways: We extend it to graph manifolds,
and we extend it to all circle bundles.

To follow this strategy through it is necessary 
to establish some statements on graph manifolds beforehand. 
These might be of independent interest. 
The main technical theorem, which is Theorem \ref{GraphRes}, describes
a character of the fundamental group, which is later used to twist Reidemeister torsion. 
This leads to the following in Theorem \ref{AlexanderThurston}.
\begin{thm*}
Let $N$ be a graph manifold with $b_1(N) \geq 2$. 
Then there exists a finite cover
$p\colon M \rightarrow N$ such that the torsion norm and
the Thurston norm coincide on $M$.
\end{thm*}

\subsection{Acknowledgements}
I would like to thank my advisor Stefan Friedl.
His insights and guidance were key for writing this article.
Furthermore, I would like to thank
Kilian Barth, Dominic Jänichen, Brendan Owens,
Mark Powell, Martin Ruderer, Sa\v so Strle,
Michael Völkl and Raphael Zentner
for helpful discussions.

I am grateful to the referee for providing many valuable suggestions and corrections.

\subsection{Conventions}
We exclusively work in the smooth setting.
By a surface we refer to an oriented compact $2$-dimensional manifold.
All $3$-manifolds we consider are irreducible, oriented and compact.
Circle bundles over $3$-manifolds are always understood to have (and admit)
a fibre-orientation.
For a natural number $k \geq 1$ we abbreviate the quotient ring 
$\IZ/k \IZ$ to $\ZkZ k$.
\section{Graph manifolds}
\subsection{Seifert fibred and graph manifolds}
In this section we recall the definition of graph manifolds and
explain how they can be simplified by taking finite covers.

Given a collection of disjointly embedded surfaces $\{\Sigma_i\}$ in a $3$-manifold $N$,
we can find disjoint bicollars, which are tubular neighbourhoods $\Sigma_i \times (-1,1) \subset N$ of $\Sigma_i$.
We denote by $N | \{\Sigma_i\} $ the manifold with boundary obtained by substituting 
each bicollar $\Sigma_i \times (-1,1)$ with $\Sigma_i \times (-1,0] \cup [0, 1)$.
\begin{defn}
	For a $3$-manifold $N$ a
	\emph{graph structure} is a collection $\mathcal{T}_N$ of
	disjointly embedded incompressible tori and a Seifert fibred structure 
	on $N|\mathcal{T}_N$.
	A \emph{graph manifold} is a connected and closed $3$-manifold 
	which admits a graph structure.
\end{defn}
We introduce some more terminology to help us deal with these manifolds. 
The elements of $\mathcal{T}_N$ are called \emph{graph tori}. 
The components of $N | \mathcal{T}_N$ are called \emph{blocks} and we denote 
the collection of blocks by $\mathcal{B}_N$.
Recall that by convention the graph tori are oriented.
If $B$ is a block of $N$ and $T$ is a boundary component of $B$, then the torus $T$
also inherits an additional orientation, namely the boundary orientation induced by $B$. 
There is a unique block $B_+(T)$ of $N$
such that these two orientations agree. The block inducing the opposite orientation
on $T$ is referred to as $B_-(T)$. The block $B_+(T)$ is \emph{behind} $T$ and
the block $B_-(T)$ is \emph{in front} of $T$. 
We will refer to the corresponding inclusions by
$\Map {j_\pm} T {B_\pm(T)}$.
It can happen that a torus $T \in \mathcal{T}_N$ bounds
the same block $B = B_\pm(T)$ on both sides. In this case we say $B$ is \emph{self-pasted} along $T$.

If the Seifert fibre structure on $N | \mathcal{T}_N$ admits a fibre-orientation
and we make the corresponding choice, then the Seifert fibres of $B_\pm(T)$ give 
rise to oriented embedded curves $\gamma_\pm \subset T$. We denote their intersection
number by $c(T):=[\gamma_+] \cdot_T [\gamma_-]$. For $T \in \mathcal{T}_N$ the number $c(T)$  
is called the \emph{fibre-intersection number}. Note that $\pm c(T)$ is well-defined
even if $N | \mathcal{T}_N$ does not admit a fibre-orientation.

\begin{defn}
A graph structure for $N$ is \emph{reduced} if all fibre-intersection numbers are
non-zero.
\end{defn}
\begin{rem}
If the fibre-intersection number of a graph torus $T$ vanishes, then
also $B_+(T) \cup_T B_-(T)$ (or $B_+(T) \cup_T$, if $B_\pm(T)$ coincide) admits a Seifert fibred structure. 
Thus for every graph structure $\mathcal{T}_N$ of $N$ there is a subcollection $\mathcal{T}_N' \subset \mathcal{T}_N$
which is the collection of graph tori of a reduced graph structure on $N$.
\end{rem}

Now we focus on how a graph manifold can be simplified by considering finite covers of it.
We can ask the following question: When can we glue a finite cover $M' \rightarrow N | \mathcal{T}_N$ 
together to a cover $M \rightarrow N$. The following property gives a sufficient condition.
\begin{defn}
Let $d \geq 1$ be a natural number.
\begin{enumerate}
	\item A finite connected cover $p\colon \widetilde T \rightarrow T$ of the two dimensional torus 
	is \emph{$d$-characteristic} if the cover is induced by the following subgroup
	\begin{align*}
		\{ g^d : g \in \pi_1(T) \} \subset \pi_1(T).
	\end{align*}
	\item A finite cover $p\colon M \rightarrow N$ of a $3$-manifold $N$ with (possibly empty) toroidal boundary  
	is \emph{$d$-characteristic} if the restriction of the cover to each boundary component
	is a $d$-characteristic cover of the torus.
\end{enumerate}
\end{defn}
\begin{lem}\label{GlueCovers}
	Let $N$ be a $3$-manifold equipped with a graph structure $\mathcal{T}_N$. 
	Let $d \geq 1$ be a natural number and $\Map {p'} {M'} {N|\mathcal{T}_N}$ be
	a $d$-characteristic cover.
	Then there is a cover $\Map p M N$ and 
	a trivial cover $\Map \phi {M|\mathcal{T}_M} {M'}$ making 
	the diagram
	\begin{center}
	\begin{tikzpicture}
		\matrix (m) [matrix of math nodes, row sep=3em, column sep=4em
					, text height=1.5ex, text depth=0.25ex]
  		{
		     M|\mathcal{T}_M && M' \\
		     &N|\mathcal{T}_N& \\
		};
		\path[->] (m-1-1) edge node [below left] {$p$} (m-2-2);
		\path[->] (m-1-1) edge node [above] {$\phi$} (m-1-3);
		\path[->] (m-1-3) edge node [below right] {$p'$} (m-2-2);
	\end{tikzpicture}
	\end{center}
	commutative, where $\mathcal{T}_M$ is the collection
	\begin{align*}
		\mathcal{T}_M := \{ C : C \text{ component of } p^{-1}(T), T \in \mathcal{T}_N\}.
	\end{align*}
\end{lem}
We suppress a proof. Its idea is that due to the cover being characteristic,
one can lift the gluing maps. 
See \cite[Theorem 2.2 and Section 4]{Hempel87} for a related discussion.
With the next lemma one can eliminate the exceptional 
Seifert fibres of a Seifert fibred manifold
by taking a suitable characteristic cover. We will use it 
to simplify the blocks of our graph manifold.
\begin{lem}[Hempel]\label{SeifertRes}
	Let $N$ be a Seifert manifold with (possibly empty) toroidal boundary. Then
	there is a $d$-characteristic cover $p\colon M \rightarrow N$ for some $d \in \IN$
	such that $M$ is diffeomorphic to a circle bundle over a surface.
\end{lem}
\begin{proof}
If $N$ has no boundary components, this can be deduced as follows.
Note that the conclusion holds for the exceptions stated in \cite[Theorem 1.17]{Aschenbrenner12}.
Thus we may assume that $N$ is a geometric $3$-manifold and not of
Sol- and hyperbolic geometry. Again the conclusion holds by \cite[Table 1]{Aschenbrenner12}.

If $N$ has a compressible boundary component, then it has to be prime
as the only non-prime Seifert fibred manifold is the connected sum $\IR P^3 \# \IR P^3$.
As $N$ is not a 2-sphere bundle over $S^1$, it is even irreducible.
A compressing disc gives rise to a embedded $2$-sphere which lies in a neighbourhood of the
disc and the boundary. This sphere has to bound
a ball and thus $N$ is diffeomorphic to a solid torus. 

If $N$ has incompressible boundary, then the statement can be found in
\cite[Lemma 4.20]{FriedlResp}.
\end{proof}

To simplify a graph manifold we will need to find covers of the
blocks. The next lemma facilitates this.
\begin{lem}\label{SurfaceBoundaryCover}
	Let $\Sigma$ be a connected surface and not the disc. 
	Let $d \geq 2$ be a natural number.
	Then there is a finite connected cover $p\colon \widetilde \Sigma \rightarrow \Sigma$ 
	which restricts to each boundary component of $\widetilde \Sigma$ to the 
	$d$-fold connected cover of a circle.
\end{lem}
\begin{proof}
	If $\Sigma$ has only one boundary component, then 
	the kernel of the Hurewicz homomorphism 
	$\pi_1(\Sigma) \rightarrow H_1(\Sigma; \ZkZ d)$ induces
	a $1$-characteristic cover such that the covering surface
	has more than one boundary component. Here we use that
	$\Sigma$ is not a disc and so not all of $\pi_1(\Sigma)$ is
	in the kernel. Now we can apply the case
	below.

	If $\Sigma$ has more than one boundary component, then the 
	kernel of the same homomorphism 
	$\pi_1(\Sigma) \rightarrow H_1(\Sigma; \ZkZ d)$ induces
	a cover with the required properties.
\end{proof}
Let $\Map p {\widetilde \Sigma} {\Sigma}$ be a finite cover which restricts 
on each boundary component of $\widetilde \Sigma$ to the $d$-fold cover of
the circle.
We refer to the cover $\Map {p \times z^d}  {\widetilde \Sigma \times S^1} {\Sigma \times S^1}$
as the \emph{associated $d$-characteristic cover}.
\begin{defn}\label{Defn:GraphTypes}
	For a graph manifold $N$ we distinguish the following types:
\begin{center}
\begin{tabularx}{\textwidth}{r X}
	Spherical type & $N$ is diffeomorphic to $S^3$.\\
	Circle type & $N$ is a non-trivial circle bundle over a surface of 
	non-positive Euler characteristic.\\
	Torus type & $N$ is a torus bundle.\\
	Composite type & $N$ admits a reduced graph structure such that 
	each block is diffeomorphic to $\Sigma \times S^1$
	with $\Sigma$ being a connected surface of negative Euler characteristic.
\end{tabularx}
\end{center}
\end{defn}
The underlying surface of a block $B \cong \Sigma \times S^1$, the space obtained by identifying
points in the same Seifert fibre, is the \emph{block surface} of $B$. If we have
fixed a trivialisation $B = \Sigma \times S^1$, we also refer to $\Sigma$ as the block
surface.
We say that $N$ is \emph{virtually} of one of the above types if there is a finite
cover $p \colon M \rightarrow N$ such that $M$ is of the type in question.
Variants of the proposition below are well-known, see \cite{Neumann97}. 
A proof is included for the convenience of the reader.
\begin{lem}\label{Lem:GraphTypes}
Let $N$ be a graph manifold. Then $N$ is virtually 
of one of the above types.
\end{lem}

\begin{proof}
We consider two cases. In the first case
$N$ admits a Seifert fibred structure and in the second 
$N$ admits a graph structure with at least one graph torus. Note that
these two cases are not disjoint.
\begin{description}
\item[Seifert fibred] The $3$-manifold $N$ admits a Seifert fibred structure.
	By Lemma \ref{SeifertRes} the manifold $N$ has a finite cover $\Map \pi E N$
	with $E$ being the total space of a circle bundle $E\rightarrow \Sigma$ over a surface 
	$\Sigma$ without boundary.

	If this circle bundle is trivial, then $E$ will be either of composite type or of torus type depending on
	the genus of $\Sigma$. As we are only considering irreducible graph manifolds, 
	the case $E \cong S^2 \times S^1$ does not appear.

	If the circle bundle is non-trivial, then either $\Sigma$ is a sphere or
	has non-positive genus. If $\Sigma$ is a sphere, then $E$ is covered by $S^3$.
	In the other case $E$ is of circle type.

\item[Graph torus] Fix a graph structure on $N$. 
	Denote the collection of tori by $\mathcal{T}_N$ and the
	collection of blocks by $\mathcal{B}_N$. By assumption $\mathcal{T}_N$ is non-empty.
	By Proposition \ref{SeifertRes} we know that each block $X \in \mathcal{B}_N$ has 
	for some $d_X \in \IN$ a finite 
	$d_X$-characteristic cover $p_X \colon \Sigma_X \times S^1 \rightarrow X$, 
	where $\Sigma_X$ is a connected surface.
	As the tori in $\mathcal{T}_N$ are incompressible, we know that $\Sigma_X$ cannot be the disc.
	Denote the least common multiple by $d := \lcm_{X \in \mathcal{B}} d_X$.
	We pick a $d/d_X$-characteristic cover 
	$\Map {q_X} {\widetilde \Sigma_X \times S^1} { \Sigma_X \times S^1}$
	by taking the associated characteristic cover of a suitable cover
	of $\Sigma_X$ given by Lemma \ref{SurfaceBoundaryCover}.
	Composing these covers we obtain
	a $d$-characteristic cover 
	\begin{align*}
	p'_X := p_X \circ q_X \colon \widetilde \Sigma_X \times S^1 \rightarrow X
	\end{align*}
	for each block $X \in \mathcal{B}$. We combine these to a $d$-characteristic cover
	$\Map {p'} {\bigcup \widetilde \Sigma_X \times S^1}{N|\mathcal{T}_N}$.
	By Lemma \ref{GlueCovers} 
	one can glue the components together and obtain a cover $p\colon M \rightarrow N$. 
	By the same lemma $M | \mathcal{T}_M \cong \bigcup \widetilde \Sigma_X \times S^1$
	and we see that $M$ admits a graph structure with each block being diffeomorphic
	to a copy of $\widetilde \Sigma_X \times S^1$.
	
	Now we will remove the block $X$ whose block surface is an annulus.
	If $\widetilde \Sigma_X$ is an annulus corresponding to a block $X \in \mathcal{B}_M$ and the two boundary
	tori of $X$ are not self-pasted, then we remove one of these two tori and extend the Seifert fibred
	structure from the removed torus to all of $X$. If by repeating this procedure we end up with a
	block whose block surface is an annulus and its boundary tori are self-pasted, then this
	has to be the only block and $M$ is a torus bundle. If not, all blocks leftover have a block
	surfaces of negative Euler characteristic.

	We continue by reducing the graph structure.
	First we remove all graph tori with vanishing 
	fibre-intersection number, obtaining the collection $\mathcal{T}_M'$. 
	We extend the Seifert fibred structure from $M|\mathcal{T}_M$ to
	$M|\mathcal{T}_M'$. This is possible, because we only removed tori $T$ with
	$c(T) = 0$. Thus if $\mathcal{T}_M'$ is the empty set, then $M$ is a Seifert fibred space.
	If $M | \mathcal{T}_M'$ admits a fibre-orientation, then $M$ will be of composite type 
	witnessed by this graph structure.
	If there is no fibre-orientation, then the fibre-orientation cover of $M | \mathcal{T}_M'$
	can be glued to a cover of $M$ as it is $2$-characteristic.
	By inspecting the induced graph structure on this cover we see that it is
	of composite type.
\end{description}
\end{proof}
In the table below we relate the above types with the Thurston geometries. 
See Scott's survey \cite{Scott83} on the geometrisation of $3$-manifolds for an introduction. 
A black dot in the table denotes that there is a manifold with the geometry given by the column 
which is of the type given by the row.
Note that most manifolds of composite type do not admit a geometry.
\begin{center}
\begin{tabular}{l | c c c c c c}
		& Spherical 	& Euclidean 	& Sol 		& Nil 		& $\mathbb{H}^2 \times \IR$ 	
		& $\widetilde {\operatorname{PSL}}(2,\IR)$\\
\hline
Spherical	& $\bullet$ 	&		&		&		&\\
Circle		&		&		& 		& $\bullet$	& 
		& $\bullet$\\
Torus 		& 		& $\bullet$	& $\bullet$	& $\bullet$\\
Composite 	&		&		&		&		& $\bullet$\\
\end{tabular}
\end{center}

\subsection{A character on graph manifolds of composite type}
In this section $N$ denotes a graph manifold of composite type
with a composite graph structure, see the definition below.
\begin{defn}
A \emph{composite graph structure} for $N$ is a reduced
graph structure $\mathcal{T}_N$ with a Seifert fibre orientation on
$N | \mathcal{T}_N$ such that each 
component of $N | \mathcal{T}_N$ is Seifert fibre preserving diffeomorphic 
to $\Sigma \times S^1$ with $\Sigma$ being a connected surface of negative Euler characteristic.
\end{defn}

Note that we are still allowing self-pastings. We can get rid of the them 
by the following lemma.
\begin{lem}\label{Selfpastings}
There is a finite cover of $N$ such that the induced composite graph structure
has no self-pastings.
\end{lem}
\begin{proof}
Let $T_1, \ldots T_k \in \mathcal{T}_N$ be the tori along which blocks are self-pasted.
A suitable cover of $N$ is the cover induced by the kernel of the map
\begin{align*}
\pi_1(N) &\rightarrow \ZkZ{2} \\
\gamma &\mapsto \sum_i \gamma \cdot [T_i].
\end{align*}
\end{proof}
Note that taking further covers with the induced graph structure cannot
reintroduce self-pastings.

\begin{lem}\label{PositiveGenus}
	Let $\Sigma$ be a connected surface of negative Euler characteristic 
	and $d \geq 3$ be a natural number. 
	Then there is a finite cover $p\colon \widetilde \Sigma \rightarrow \Sigma$
	such that $\widetilde \Sigma$ has positive genus and $p$ restricts
	to the $d$-fold connected cover on each boundary component.
\end{lem}
\begin{proof}
	Pick a connected cover $p_\Sigma\colon \widetilde \Sigma \rightarrow \Sigma$ which
	restricts to the $d$-fold connected cover on each boundary component.
	Its existence is guaranteed by 
	Lemma \ref{SurfaceBoundaryCover}. Denote the degree of $p_\Sigma$ by $n$.
	The genus $g(\Sigma)$ and the Euler characteristic $\chi(\Sigma)$ 
	are related by the equation
	\begin{align*}
		\chi(\Sigma) = 2 - 2g(\Sigma) - b_0 (\partial \Sigma).
	\end{align*}
	The Euler characteristic is multiplicative under finite covers. Therefore we obtain
	the equations
	\begin{align*}
	2 - 2g (\widetilde \Sigma) - b_0 (\partial \widetilde \Sigma) &= 
	n \left (2 - 2g(\Sigma) - b_0 (\partial \Sigma) \right)\\
	\Rightarrow 2 - 2g (\widetilde \Sigma) -  \left( n / d\right)  b_0 (\partial \Sigma)  &=
	 n \left (2 - 2g(\Sigma) - b_0 (\partial \Sigma) \right)\\
	\Rightarrow	2 - 2g (\widetilde \Sigma) &= 
	n \left (2 - 2g(\Sigma) - \frac{d-1}{d} \cdot b_0 (\partial \Sigma) \right)
	\end{align*}
	If $g(\Sigma) > 0$ holds, then we obtain the inequality $g(\widetilde \Sigma) \geq g(\Sigma) > 0$.
	Otherwise $b_0(\partial \Sigma) \geq 3$ has to hold because $\Sigma$ has negative Euler characteristic.
	Thus the right hand side of the above
	equation is non-positive and so $g(\widetilde \Sigma) > 0$.
\end{proof}

\begin{lem}\label{PosCover}
	The manifold $N$ has a finite cover $p\colon \widetilde N \rightarrow N$
	such that the induced graph structure for $\widetilde N$ is reduced
	and all block surfaces have positive genus.
\end{lem}
\begin{proof}
	By Lemma \ref{PositiveGenus} there is a $3$-characteristic 
	cover $p_X\colon \widetilde X \rightarrow X$ for every block $X \in \mathcal{B}$
	such that $\widetilde X$ is diffeomorphic to $\Sigma_{\widetilde X} \times S^1$ 
	with $\Sigma_{\widetilde X}$ being a surface
	of positive genus. By Lemma \ref{GlueCovers} we can glue these covers
	together and obtain a finite cover $p\colon \widetilde N \rightarrow N$ 
	with the induced graph structure. This graph structure is reduced as
	for every graph torus $T \in \mathcal{T}$ and every component $\widetilde T$ of $p^{-1}(T)$ 
	we have the equality $c(\widetilde T) = c(T)$. 
	The condition on the block surfaces holds by construction.
\end{proof}

The next lemma shows how one can extend a cohomology class 
with coefficients in $\IZ_k$ from the boundary of a block to the block itself by 
passing to a finite cover. This is not possible for integral cohomology classes.
\begin{lem}\label{ClassExtension}
	Let $X:= \Sigma \times S^1$ be the trivial circle bundle over a connected surface
	of positive genus. Fix a natural number $k \geq 1$.
	Let $\alpha_{\partial X} \in H^1( \partial X ; \ZkZ k)$ be a cohomology class
	which evaluates to $1$ on the $S^1$-factor, i.e. $\alpha_{\partial_X}$ satisfies
	\begin{align*}
		\langle \alpha_{\partial X}, [\{x\} \times S^1] \rangle = 1
	\end{align*}
	for all $x \in \partial \Sigma_X$.
	Then there is a $1$-characteristic finite cover 
	$\Map p {\widetilde X} X$ and a class 
	$\alpha_{\widetilde X} \in H^1( \widetilde X; \ZkZ k)$ such that
	\begin{align*}
		p^* \alpha_{\partial X} = i^* \alpha_{\widetilde X} 
		\in H^1( \partial \widetilde X; \ZkZ k),
	\end{align*}
	where $i\colon \partial \widetilde X \rightarrow \widetilde X$
	denotes the inclusion of the boundary.
\end{lem}
\begin{proof}
	The projections of the product $\partial X=\partial \Sigma \times S^1$
	induce an isomorphism $H^1(\partial X; \ZkZ k) \cong H^1(\partial \Sigma; \ZkZ k) \oplus H^1(S^1; \ZkZ k)$.
	Denote by $\beta_{\partial \Sigma} + \theta$ the image of $\alpha_{\partial X}$ 
	under this isomorphism.
	We refer by $p_\Sigma \colon \widetilde \Sigma \rightarrow \Sigma$ to the cover
	corresponding to the kernel of the homomorphism
	\begin{align*}
		\pi_1( \Sigma ) \rightarrow H_1(\Sigma, \partial \Sigma;\ZkZ k).
	\end{align*}
	Its restriction $p_{\partial \Sigma}$ to the boundary of $\widetilde \Sigma$ is the
	trivial disconnected cover of order  $|H_1(\Sigma, \partial \Sigma;\ZkZ k)|$.
	We prove that there is a class $\beta_{\widetilde \Sigma} \in H^1(\widetilde \Sigma; \ZkZ k)$ 
	such that 
	\begin{align*}
		i^* \beta_{\widetilde \Sigma} = p_{\partial \Sigma}^* \left( \beta_{\partial \Sigma} \right)
	\end{align*}
	Note that the only obstruction for the existence of such a $\beta_{\widetilde \Sigma}$
	is that
	\begin{align*}
		\langle p_{\partial \Sigma}^* \left( \beta_{\partial \Sigma}\right) , \partial [\widetilde \Sigma] \rangle = 0
	\end{align*}
	as can be seen from the long exact sequence of the pair 
	$(\widetilde \Sigma, \partial \widetilde \Sigma)$. But as the cover multiplies the boundary
	components we have
	\begin{align*}
		\langle p_{\partial \Sigma}^* \left( \beta_{\partial \Sigma}\right) , \partial [\widetilde \Sigma] \rangle  &=\langle   \beta_{\partial \Sigma} , {p_{\partial \Sigma}}_* \partial [\widetilde \Sigma] \rangle\\
		&= \langle   \beta_{\partial \Sigma} , {\partial p_{\Sigma}}_*  [\widetilde \Sigma] \rangle  \\
		&=|H_1(\Sigma, \partial \Sigma;\ZkZ k)| \langle   \beta_{\partial \Sigma}, \partial [ \Sigma] \rangle\\
		&= 0 \in \ZkZ k
	\end{align*}
	as $|H_1(\Sigma, \partial \Sigma;\ZkZ k)|$ is divisible by $k$. We denote the characteristic cover 
	associated to $p_\Sigma$ by $p := p_\Sigma \times S^1$.
	It is $1$-characteristic and the element $\alpha_{\widetilde X} \in H^1(\widetilde X; \ZkZ k)$ 
	mapping to  $\beta_{\widetilde \Sigma} + \theta$ has the required properties.
\end{proof}

Now we are able to prove the main theorem of this section and
one of the main technical results of this article.
\begin{thm}\label{GraphRes}
	Let $N$ be a graph manifold of composite type.
	If $d \geq 2$ is coprime to all fibre-intersection numbers of $N$, then there is
	\begin{enumerate}
	\item a finite cover  $p\colon M \rightarrow N$, 
	\item a composite graph structure on $M$,
	\item a cohomology class $\alpha\in H^1(M; \ZkZ{d})$ 
	such that for each block $B \in \mathcal{B}_M$ the class $\theta_B$ of a
	Seifert fibre evaluates to
	\begin{align*}
		\langle \alpha, \theta_B \rangle &= 1 \in \ZkZ d.
	\end{align*}
	\end{enumerate}
\end{thm}
\begin{proof}
	By Lemma \ref{Selfpastings} we may assume that $N$ has no self-pastings.
	Furthermore, by Lemma~\ref{PosCover} we assume that all block surfaces have
	positive genus.

	We construct the class $\alpha$ in a finite cover.
	Define $\alpha_T \in \Hom_\IZ \left(H_1(T;\IZ), \ZkZ d \right) 
	= H^1(T; \ZkZ d)$ by the equation:
	\begin{align*}
		\langle \alpha_T, \beta \rangle 
		:= c(T)^{-1}\left( [\gamma_+] \cdot \beta - [\gamma_-] \cdot \beta\right).
	\end{align*}
	Here we have used that each number $c(T)$ is coprime to $d$ and thus
	$c(T)$ is invertible in $\ZkZ d$.
	It has the defining property that $\langle \alpha_T, [\gamma_\pm] \rangle = 1$.
	
	For each block $B \in \mathcal{B}_N$ we do the following:
	There is a unique class $\alpha_{\partial B} \in H^1(\partial B; \ZkZ d)$
	such that $\alpha_{\partial B}$ restricts to $\alpha_T$ on $T \subset \partial B$.
	By Lemma \ref{ClassExtension} there is a $1$-characteristic cover 
	$p_B\colon \widetilde B \rightarrow B$ and a class $\alpha_{\widetilde B} \in H_1 \big (\widetilde B; \ZkZ d \big)$
	extending $p^*_B \alpha_{\partial B}$. 
	We combine the covers $p_B$ with Lemma \ref{GlueCovers} to a cover $p\colon M \rightarrow N$.
	The manifold $M$ is equipped with the induced composite graph structure.

	If $\theta_{\widetilde B}$ is the class of a Seifert fibre in the block 
	$\widetilde B \in \mathcal{B}_M$, then $\theta_{\widetilde B} = \theta_{\partial \widetilde B}$ with $\theta_{\partial \widetilde B}$ being
	the class of a Seifert fibre in the boundary of $\widetilde B$. Assuming $\widetilde B$ covers the block $B$, we obtain
	\begin{align*}
		\langle \alpha_{\widetilde B}, \theta_{\widetilde B} \rangle 
		&= \langle \alpha_{\widetilde B}, \theta_{\partial \widetilde B}\rangle 
		= \langle p_B^* \alpha_{\partial B}, \theta_{\partial \widetilde B}\rangle = 1,
	\end{align*}
	where in the last equality we used that $p_B$ is $1$-characteristic.

	The last step is to combine the classes $\alpha_{B} \in H^1(B; \ZkZ d)$ with $B \in \mathcal{B}_M$
	to a class $\alpha \in H^1(M; \ZkZ d)$. For this we are going to use the
	Mayer-Vietoris exact sequence far below. Define maps $\phi_B$ for
	each block $B$ by
	\begin{align*}
		\phi_{B} \colon H^1(B; \ZkZ d) &\rightarrow \bigoplus_{T \subset \partial B} H^1(T; \ZkZ d)\\
			\alpha &\mapsto \Sigma_{T \subset \partial B}
						\begin{cases} 
						j_T^*\alpha & T \text{ carries the boundary orientation of $B$}\\
						-j_T^* \alpha & \text{ otherwise}
						\end{cases},
	\end{align*}
	where $j_T \colon T \hookrightarrow B$ is the inclusion.
	Denote by $\phi$ the map 
	\begin{align*}
		\oplus_{B \in \mathcal{B}_M} \phi_B \colon
		\bigoplus_{B \in \mathcal{B}_M} H^1(B; \ZkZ d)
		\rightarrow \bigoplus_{T \in \mathcal{T}_M} H^1(T; \ZkZ d).
	\end{align*}
	The corresponding Mayer-Vietoris sequence is the exact sequence
	\begin{align*}
		\bigoplus_{T \in \mathcal{T}_M} H^1(T; \ZkZ d)
		\overset{\phi}{\leftarrow} \bigoplus_{B \in \mathcal{B}_M} H^1(B; \ZkZ d)
		\leftarrow H^1(M; \ZkZ d).
	\end{align*}
	Note we have the equality
	\begin{align*}
		\phi \Big(\sum_{B \in \mathcal {B}_M} \alpha_B \Big)
		 = \sum_{T \in \mathcal {T}_M} j_T^*\alpha - j_T^*\alpha
		 = 0.
	\end{align*}
	Thus there is an $\alpha \in H^1( M; \ZkZ d)$ with the property that $\alpha$ restricted
	to a block $B \in \mathcal{B}_M$ is $\alpha_{B}$. 
	The class $\alpha$ evaluates correctly on the class $\theta_B$ of a Seifert fibre in
	a block $B$ as 
	$\langle \alpha, \theta_B \rangle = \langle \alpha_{B}, \theta_B \rangle = 1$.
\end{proof}

\section{The Thurston Norm}
Let $N$ be a $3$-manifold with toroidal boundary (if any).
Thurston \cite{Thurston86} introduced the following complexity for surfaces.
It is better behaved than the Euler characteristic and still multiplicative 
under finite covers. Before we come to the next definition, recall our convention that all surfaces are oriented. 
\begin{defn}\label{ThurstonNorm}\label{SurfacesComplexity}
\begin{enumerate}
\item  The \emph{complexity} of a surface $\Sigma$ with components $\Sigma_i$ is
\begin{align*}
	\chi_-(\Sigma) &:= \sum_{i} \max(0, -\chi(\Sigma_i)).
\end{align*}
\item For a homology class $\sigma \in H_2(N, \partial N; \IZ)$, the
	\emph{Thurston norm} of $\sigma$ is
	\begin{align*}
	\Thur \sigma :=\min \{ \chi_-(\Sigma) : \Sigma \text{ a properly embedded surface with } \sigma = [\Sigma]\}.
	\end{align*}
	For a class $\beta \in H^1(N; \IZ)$ the Poincaré duality map 
	$\PD\colon H^1(N; \IZ) \rightarrow H_2(N, \partial N; \IZ)$ is used to transfer the norm:
	\begin{align*}
		\Thur \beta := \Thur {\PD(\beta)}.
	\end{align*}
\end{enumerate}
\end{defn}
\begin{rem}
Thurston \cite[Theorem 1]{Thurston86} showed that $\Thur .$ is a semi-norm, 
i.e.  for all $n \in \IZ$ and all $\sigma, \rho \in H_2(N;\IZ)$ one has
$\Thur {\sigma + \rho} \leq \Thur \sigma + \Thur \rho$ and $\Thur {n \sigma} = |n| \Thur \sigma$.
The last equation implies that we can extend it to a semi-norm on $H_2(N; \partial N; \IQ)$.
\end{rem}
\begin{prop}[Eisenbud-Neumann]\label{ThurstonSplit}
	Let $\mathcal{T}$ be a collection of incompressible embedded tori in $N$.
	Denote by $\mathcal{B}$ the collection of components of $N|\mathcal T$
 	and by $i_B \colon B \hookrightarrow N$ the inclusion of a component $B \in \mathcal{B}$.
	Then the Thurston norm of each $\beta \in H^1(N;\IZ)$ satisfies the equality
	\begin{align*}
		\Thur \beta = \sum_{B \in \mathcal B} \Thur {i_B^* \beta }.
	\end{align*}
\end{prop}
\begin{proof}
The argument given in \cite[Proposition 3.5]{Eisenbud85} not only works for tori of a
JSJ-decomposition but for any collection of incompressible tori.
\end{proof}

\begin{prop}\label{ThurstonProduct}
	Let $\Sigma$ be a connected surface of negative Euler characteristic.
	The Thurston norm on $\Sigma \times S^1$ fulfils
	\begin{align*}
		\Thur \sigma = -\chi(\Sigma) \cdot |\langle \sigma, [S^1]\rangle|
	\end{align*}
	for all $\sigma \in H^1(\Sigma \times S^1; \IZ)$.
\end{prop}
\begin{proof}
	The Eilenberg-MacLane space $\Eilen(\IZ,1)$ is an oriented $S^1$. 
	We write $\theta \in H^1(\Eilen(\IZ,1), \IZ) = \Hom(\IZ, \IZ)$ for the
	element corresponding to the identity.
	For every manifold $X$ we can associate to each cohomology 
	class $\sigma \in H^1(X; \IZ)$ a smooth map 
	$f_\sigma : X \rightarrow S^1$ such that $\sigma = f_\sigma^* \theta$.
	
	First we prove that on each $\beta \in H^1(\Sigma; \IZ) \subset H^1(\Sigma \times S^1; \IZ)$
	the Thurston norm vanishes. Observe that preimages of a regular value
	of the map	
	\begin{align*}
		g: \Sigma \times S^1 &\rightarrow S^1 \\
		 	(x,z) &\mapsto f_\beta(x)
	\end{align*}
	are Poincaré dual to $\beta$. But all of them are either empty or tori.
	Thus the Thurston norm of $\beta$ vanishes, i.e. $\Thur \beta = 0$.

	Let $\beta + \lambda \theta$ be any element of $H^1(\Sigma;\IZ) \oplus H^1(S^1;\IZ)
	= H^1(\Sigma \times S^1;\IZ)$. Using the reverse triangle inequality, we obtain
	\begin{align*}
		\Thur { \beta + \lambda \theta } = |\lambda| \Thur \theta.
	\end{align*}
	But $\theta$ is Poincaré dual to the fibre of the fibre bundle
	\begin{align*}
		\Sigma \times S^1 &\rightarrow S^1\\
		(x, z) &\mapsto z.
	\end{align*}
	By \cite[Section 3]{Thurston86} a fibre of a fibre bundle is
	minimising the Thurston norm. Therefore the Thurston norm is $\Thur \theta = -\chi(\Sigma)$.
	We obtain the equation
	 \begin{align*}
	 	\Thur{\beta + \lambda \theta} = -|\lambda| \cdot \chi(\Sigma) 
		=-\chi(\Sigma) \cdot |\langle\beta + \lambda \theta, [S^1]\rangle|.
	\end{align*}
	\end{proof}
\section{Maximal abelian torsion}
We give an outline of the following section.
The main statement is Theorem~\ref{AlexanderThurston} showing
that a graph manifold admits
a cover in which the torsion norm, which will be defined in Definition~\ref{defn:TorsionNorm},
agrees with the Thurston norm. 

The most interesting case will be that of a graph manifold of
composite type. In this case the result will follow from a computation
of twisted Reidemeister torsion, which is defined in Definition~\ref{defn:twReidem}.
More exactly, we will find a cover~$M$ such that we can compute explicitly
the Reidemeister torsion~$\tau(M, \IC(t)_{\alpha\otimes \sigma})$,
twisted by the representation $\IC(t)_{\alpha\otimes \sigma}$, described in Construction~\ref{Constr:Twist}.
We will see that in this cover the equality $ \Thur \sigma = \width \tau(M, \IC(t)_{\alpha\otimes \sigma}) $ holds.
The claim that the torsion norm~$\TorN \sigma$ and the Thurston norm~$\Thur \sigma$ of
a class~$\sigma$ agree, then follows directly from the inequality stated in Proposition~\ref{Comparison}
\[ \Thur \sigma \geq \TorN \sigma \geq \width \tau(M, \IC(t)_{\alpha\otimes \sigma}). \]
\subsection{Torsion of chain complexes}
We are going to review the notion of torsion and
fix conventions. 
For this we follow \cite{Turaev01}, which
uses the multiplicative inverse convention of
the article \cite{Milnor66}.
Fix a field $\IK$ for the rest of the section.
\begin{defn}
Let $V$ be a $\IK$-vector space and 
$b = \{b_1, \ldots, b_k\}$, $c = \{c_1, \ldots, c_k\}$ two 
ordered bases of $V$.
The \emph{base change matrix} expressing $b$ in $c$ is the matrix 
$h \in \GL{k,\IK}$ given by
\begin{align*}
b_i = \sum_{j} h_{ij} c_j.
\end{align*}
We write $[b/c]:= \det h \in \IK^*$. 
\end{defn}

\begin{lem}
Let $(V,c)$, $(A,a)$, $(B,b)$ be based $\IK$-vector spaces and 
\begin{align*}
	0 \rightarrow A \rightarrow V \rightarrow B \rightarrow 0
\end{align*}
a short exact sequence. Let $\Map s B V$ be a section.
Then the list $a \cup s(b)$ is a basis of $V$ and
$[a \cup s(b)/c] \in \IK$ is independent of the choice of the section $s$.
\end{lem}
\begin{proof}
\cite[Section I.1 p. 2]{Turaev01}.
\end{proof}
Assume that we are given a based chain complex $C$. This
means $C$ is a chain complex, whose chain modules are based 
$\IK$-vector spaces $(C_i, c_i)$.
There is no condition on the boundary morphism with respect to the basis.
If we can arrange $C$ to be even acyclic, then for every degree $i$ we have 
the short exact sequence
\begin{align*}
	0 \rightarrow \Ima \partial_{i+1} \rightarrow C_i \rightarrow \Ima \partial_{i} \rightarrow 0.
\end{align*}
We pick a basis $b_i$ of $\Ima \partial_{i}$ and a split
$\Map {s_i} {\Ima \partial_{i}} {C_i}$. One can show
that 
\begin{align*}
	\tau(C) := \prod_i [b_{i+1} \cup s_i( b_{i}) /c_i ]^{{(-1)}^{i+1}} \in \IK^*
\end{align*}
is independent of all the choices made, see \cite[Lemma 1.3]{Turaev01}
\begin{defn}
Let $C$ be an acyclic based chain complex. The element
$\tau(C) \in \IK^*$ constructed above is called the \emph{torsion}
of $C$.
\end{defn}
If $c'$ and $c$ are two different bases for the acyclic chain complex $C$,
then the torsion of $C$ changes as follows when we switch from $c$ to $c'$: 
\begin{align*}
	\tau(C,c') =& [c/c']\cdot\tau(C,c) \in \IK^*\\
	\text{with } [c/c'] :=& \prod_i  [c_i/c'_i]^{(-1)^{i+1}}.
\end{align*}
It follows that the
torsion considered up to a sign is independent of the ordering 
of the basis, see \cite[Remark 1.4.1]{Turaev01}.

\subsection{Torsion of $3$-manifolds}
We continue by defining the twisted Reidemeister torsion
of a $3$-manifold. This invariant will then be related later
to the Thurston norm.

We assume that $M$ is equipped with a CW-structure, although
the twisted Reidemeister torsion will turn out to be independent
of this choice as well. Here the push-out data, not just
the filtration, will be part of the CW-structure. 
Furthermore, let us fix a
universal cover $\Map p {\widetilde M} M$. The fundamental group $\pi_1 := \pi_1(M)$
acts by deck transformations on $\widetilde M$.
We equip the manifold $\widetilde M$ with the induced CW-structure.
This gives the cellular complex of $\widetilde M$ the
structure of a chain complex of a free left $\IZ[\pi_1]$-modules. 
To calculate the torsion we need to equip this chain complex with a basis. 
This involves a choice.
\begin{defn}
\begin{enumerate} 
\item A \emph{fundamental family} for $M$ is a family $\mathfrak{e}$ of cells
	of $\widetilde M$ such that each cell of $M$ is covered by exactly one 
	cell in $\mathfrak{e}$.
\item The complex of $\IZ[\pi_1]$-modules of $\widetilde M$ described above
	equipped with the basis coming from a fundamental family $\mathfrak{e}$ 
	is denoted by $C(M, \mathfrak{e})$. From now on we exclusively refer to
	this chain complex as the \emph{cellular complex}.
\end{enumerate}
\end{defn}
Now we fix a fundamental family $\mathfrak{e}$ for $M$ and 
twist the complex $C(M, \mathfrak{e})$ with a suitable $(\IK, \IZ[\pi_1])$-bimodule $A$.
Pick a basis $\{a_1, \ldots, a_n\}$ of the underlying 
$\IK$-vector space of $A$. We equip the $\IK$-vector space 
$A \otimes_{\IZ[\pi_1]} C_k(M, \mathfrak{e})$ with the basis
\begin{align*}
	\{a_i \otimes e : e \in \mathfrak{e} \text{ a $k$-cell, } 1 \leq i \leq n  \}
\end{align*}
Now the chain complex $A \otimes_{\IZ[\pi_1]} C(M,\mathfrak{e})$ is based.
If additionally this complex is acyclic, we can define its torsion.
Through the bimodule structure of $A$, each element 
$g \in \pi_1$ gives rise to a $\IK$-linear map $\Map {g_A} A A$ of
the $\IK$-vector space $A$. The abelian group $\W A$ is defined by 
\begin{align*}
	\W A := \IK^*/ \langle \pm \det g_A : g \in \pi_1 \rangle.
\end{align*}
\begin{defn}\label{defn:twReidem}
	Suppose the complex $A \otimes_{\IZ[\pi_1]} C(M, \mathfrak{e})$ is acyclic, then
	\begin{align*}
		\tau(M, A) := \tau(A \otimes_{\IZ[\pi_1]} C(M, \mathfrak{e})) \in \W A 
	\end{align*}
	is the $A$-\emph{twisted Reidemeister torsion} of $M$.
\end{defn}
\begin{rem}
An elementary calculation shows that $\tau(M,A)$ is independent of the choice of basis of $A$.
This would also be true without considering it to be an element of $W(A)$.
As the notation suggests $\tau(M,A)$ is also independent of the choice of a fundamental family, 
see \cite[II.6.1]{Turaev01}. For this to hold, it was necessary to pass from $\IK^*$ to $W(A)$.
Furthermore, it is even independent of the choice
of a CW-structure by \cite[Theorem IV]{Kirby69}.
\end{rem}

The following bimodules will play a special role as they 
provide a link between the Thurston norm and further invariants
defined below.
\begin{constr}\label{Constr:Twist}
Denote by $\IC(t)$ the quotient field of the polynomial ring in one-variable.
Given a character $\Map \alpha {\pi_1(M)} {\ZkZ k} \subset \IC^*$ and 
a cohomology class $\sigma \in H^1(M;\IZ)$ on $M$, we define the following
ring homomorphism:
\begin{align*}
\Map { \rho_{\alpha\otimes \sigma} } {\IZ[\pi_1]&} {\IC(t)}\\
g &\mapsto \alpha(g) t^{\sigma(g)}
\end{align*}
Such a ring homomorphism gives $\IC(t)$ the structure
of a $(\IC(t), \IZ[\pi_1])$-bimodule. This bimodule we denote
by $\IC(t)_{\alpha\otimes \sigma}$.
\end{constr}

Note that the $\IZ[\pi_1]$-action on $\IC(t)_{\alpha\otimes \sigma}$ acts through
the scalars of the $\IC(t)$-vector space. This ensures that we do not lose
much information by the transition from $\IC(t)^*$ to $\W {\IC(t)_{\alpha\otimes \sigma}}$.
The following function is able to extract valuable information from the torsion.
\begin{defn}
	The function $\Map \width {\IC(t)} {\IZ \cup \{-\infty\}}$
	is the unique function such that for all $p(t), q(t) \in \IC(t)^*$
	we have
	\begin{align*}
		\width& (0) = -\infty \\
		\width& \left(\sum_{i=a}^b c_i t^i \right) = b-a \text{ if } c_a, c_b \neq 0 \\
		\width& \left( \frac{p(t)}{q(t)} \right) = \width p(t) - \width q(t).
	\end{align*}
\end{defn}
The next lemma shows that we can use the function $\width$ to extract information
from the twisted Reidemeister torsion.
\begin{lem}
The function $\Map \width {\IC(t)} {\IZ \cup \{-\infty \}}$ fulfils:
\begin{enumerate}
	\item For all $x,y \in \IC(t)$ we have the equality
	\begin{align*}
		\width(xy) &= \width(x) + \width(y).
	\end{align*}
	\item The restriction $\width\colon \IC(t)^* \rightarrow \IZ$ to the units of $\IC(t)$ 
	descends to the group 
	$\W {\IC(t)_{\alpha\otimes \sigma}}$.
\end{enumerate}
\end{lem}
\begin{proof}
\begin{enumerate}
\item It is enough to check it for $x,y \in \IC[t]$ and then it can be
	deduced from the fact that $\IC$ does not have
	zero-divisors.
\item The second property follows from the fact that
	for all $p(t) \in \IC(t)^*$ and $g \in \pi_1$
	we get
	\begin{align*}
		\width (p(t) \cdot \det g_A)
		=\width (p(t)\alpha(g)t^{\sigma(g)})
		&=\width p(t)+ \width (\alpha(g)t^{\sigma(g)})\\
		&=\width p(t).
	\end{align*}
\end{enumerate}
\end{proof}

\subsection{Torsion of composite type graph manifolds}\label{CompositeFolds}
In this section, we will compute the twisted Reidemeister torsion,
for the representation~$\IC(t)_{\alpha \otimes \sigma}$, of
graph manifolds of composite type. We have to make further assumptions 
on the graph manifold but they can be arranged to hold by 
passing to a finite cover.

Let $N$ be a graph manifold of composite type 
and $b_1(N) \geq 2$. We pick a composite graph structure for $N$.
By taking a finite cover we may assume that $N$ has no self-pastings.

The homology class of the Seifert fibres in a block $B$ is denoted 
by $\theta_B \in H_1(B; \IZ)$.
Furthermore, denote the block surface of a block $B$ by $\Sigma_B$.
Additionally, we fix a cohomology class $\alpha \in H^1(N; \ZkZ k)$
for a natural number $k \geq 2$ which evaluates on every Seifert fibre in every block $B$ 
to $\alpha(\theta_B) = 1$.
By Theorem \ref{GraphRes} this situation can be achieved by taking a further finite cover.

We denote the collection of tori of
the graph decomposition by $\mathcal{T}_N$ and the collection of blocks
by $\mathcal{B}_N$.  Fix a CW-structure on $N$ such that the tori in $\mathcal{T}_N$
and the blocks in $\mathcal{B}_N$ 
become subcomplexes. As before we choose a universal cover $\Map p {\widetilde N} N$
and a fundamental family $\mathfrak{e}$ of $N$.

We have constructed the module $\IC(t)_{\alpha \otimes \sigma}$
in Construction~\ref{Constr:Twist}.
We are going to calculate $\tau(N, \IC(t)_{\alpha\otimes \sigma})$ for 
an arbitrary class $\sigma \in H^1(N; \IZ)$.
The Reidemeister torsion splits along the graph structure, which
makes calculations feasible.
This is made precise in Lemma \ref{TorsionSplit} below. 

We abbreviate $\IC(t)_{\alpha \otimes \sigma}$ by $A$. The blocks in $\mathcal{B}_N$ 
are CW-subcomplexes whose torsion is easily computable. 
The following allows us to relate the torsion of the blocks
to the torsion of all of $N$.
\begin{defn}
Let $B \subset N$ be a subcomplex of the CW-structure on $N$.
\begin{enumerate}
\item The CW-complex $p^{-1}(B)$ inherits a left $\IZ[\pi_1(N)]$-action.
We denote the subcollection of cells in $\mathfrak{e}$ which are contained
in $p^{-1}(B)$ by $\mathfrak{e}|_B$. The corresponding \emph{cellular complex}
of $\IZ[\pi_1(N)]$-modules equipped with $\mathfrak{e}|_B$ as a basis
is denoted by $C(B \subset N, \mathfrak{e}|_B)$.
\item Suppose $A \otimes_{\IZ[\pi_1(N)]} C(B \subset N, \mathfrak{e}|_B)$ is acyclic.
The corresponding torsion is denoted by
\begin{align*}
	\tau(B \subset N, A) := \tau\left(A \otimes_{\IZ[\pi_1]} C(B \subset N, \mathfrak{e}|_B)\right)
\end{align*}
\end{enumerate}
\end{defn}

We are going to construct a Mayer-Vietoris sequence which shows that the torsion
splits over the pieces of our graph decomposition. 
Define the following map of chain complexes:
\begin{align*}
	\phi_T \colon C(T \subset N) 
		&\rightarrow \bigoplus_{B \in \mathcal{B}_N} C(B \subset N)\\ 
	e &\mapsto (j_+)_* e -  (j_-)_* e,
\end{align*}
where $j_\pm \colon p^{-1}(T) \hookrightarrow p^{-1}(B_\pm(T))$ denotes the inclusion of the preimage
of the boundary torus $T$ into the preimage of the block $B_\pm(T)$.

\begin{lem}\label{TorsionSplit}
	\begin{enumerate}
	\item The short sequence of cellular complexes 
	\begin{align*}
	0 \rightarrow \bigoplus_{T \in \mathcal{T}} C(T \subset N, \mathfrak{e}|_T) \rightarrow \bigoplus_{X \in \mathcal{B}} C(X \subset N, \mathfrak{e}|_X) 
	\rightarrow C(N, \mathfrak{e}) \rightarrow 0
	\end{align*}
	is exact and splits. The first map is the sum $\oplus_T \phi_T$ and the second is induced by the inclusions.
	\item If $A \otimes_{\IZ[\pi_1]} C(T \subset N, \mathfrak{e}|_T)$ and $A \otimes_{\IZ[\pi_1]} C(B \subset N, \mathfrak{e}|_B)$ are all acyclic,
	then so is $A \otimes_{\IZ[\pi_1]} C(N, \mathfrak{e})$ and the equation
	\begin{align*}
		\tau(N,A) \prod_{T \in \mathcal{T}}\tau(T \subset N,A) = \prod_{X \in \mathcal{B}} \tau(X \subset N, A)
	\end{align*}
	holds in $W(A)$.
	\end{enumerate}
\end{lem}
\begin{proof}
\begin{enumerate}
	\item The exactness can be obtained by noting that two distinct blocks intersect exactly 
	in a finitely many graph tori and three and more distinct blocks intersect in the empty set.
	The complex $C(N, \mathfrak{e})$ is free and therefore the short exact sequence splits. 
	\item We need to construct a lift of the basis $\mathfrak{e}$
	to $\bigoplus_{X \in \mathcal{B}} C(X \subset N, \mathfrak{e}|_X)$. 
	Enumerate the blocks 
	$\mathcal{B}_N = \{B_1, \ldots B_n\}$ arbitrarily
	and define 
	\begin{align*}
	r(e) := e \in C(B_k \subset N),
	\end{align*}
	where $k$ is the smallest $k$ such that $e$ is a cell of $B_k$.
	Note that the short exact sequence stays exact after 
	tensoring with $A$ as it splits.
	The corresponding long exact sequence shows the claim about the acyclicity 
	of $A \otimes_{\IZ[\pi_1]} C(N, \mathfrak{e})$.
	Torsion behaves well under short exact sequences given that the bases are compatible, 
	see \cite[Theorem 3.4]{Turaev01}.  
	We calculate the base change determinant
	\begin{align*}
		\left[ \bigcup_{T\in \mathcal{T}} e|_T \cup r(e_N) \middle/ \bigcup_{B \in \mathcal{B}_N} e|_B \right] = 1 \in W(A),
	\end{align*}
	which is $1$ as all the base change matrices are invertible and only have entries in $\IZ$. It follows from
	\cite[Theorem 3.4]{Turaev01} that we have the equality
	\begin{align*}
		\tau(N,A) \prod_{T \in \mathcal{T}}\tau(T \subset N,A) 
			= \prod_{B \in \mathcal{B}_N} \tau(B \subset N, A).
	\end{align*}
\end{enumerate}
\end{proof}

As the tori in $\mathcal{T}_N$ are incompressible, the fundamental
group $\pi_1(B)$ for every block $B \in \mathcal{B}_N$ injects into $\pi_1(N)$. 
Therefore a component of $p^{-1} (B) \subset \widetilde N$ is a universal
cover of $B$. Fix for every block $B$ a component $C_B$ in $p^{-1}(B)$ and
denote the inclusion by $j_B \colon C_B \subset p^{-1}(B)$. With the
projection of $p^{-1}(B)$ onto $C_B$ we obtain a fundamental family $\mathfrak{e}_B$
for $B$ from the fundamental family of $N$.

Recall that we have denoted the homology class of the Seifert fibres in a block $B$
by $\theta_B \in H_1(B; \IZ)$ and the block surface of $B$ by $\Sigma_B$.
\begin{lem}\label{TorsionBlock}
	The complex $A \otimes_{\IZ[\pi_1]} C(B \subset N, \mathfrak{e}|_B)$ is acyclic and its
	torsion is 
	\begin{align*}
		\tau(B \subset N, A) = \left(1- \alpha(\theta_B)t^{\sigma(\theta_B)} \right)^{-\chi(\Sigma_B)} \in W(A)
	\end{align*}
\end{lem}
\begin{proof}
\begin{enumerate}
\item We first relate the chain complexes $C(B, \mathfrak{e}_B)$ with $C(B \subset N, \mathfrak{e}|_B)$.
Note that the following map is an isomorphism of chain complexes.
\begin{align*}
	\IZ[\pi_1(N)] \otimes_{\IZ[\pi_1(B)]} C(B, \mathfrak{e}_B) &\rightarrow C(B \subset N, \mathfrak{e}|_B)\\
	g \otimes e &\mapsto g \cdot (j_B)_* e
\end{align*}
The elements of the corresponding bases map to each other up to $\pi_1(N)$-translations.
But these are modded out when we go over to $W(A)$.
We get the following
\begin{align*}
	q( \tau(B, A\otimes_{\IZ[\pi_1(N)]} \IZ[\pi_1(B)])) = \tau(B \subset N, A) \in W(A),
\end{align*}
where $q$ denotes the quotient map  $q\colon W(A\otimes_{\IZ[\pi_1(N)]} \IZ[\pi_1(B)]) \rightarrow W(A)$.
\item In \cite[VII.5.2]{Turaev02} the twisted Reidemeister torsion of a fibred
manifold is calculated. Our manifold is fibred with monodromy $\id_{\Sigma}$. 
The upshot is that in our situation $A \otimes C(B, \mathfrak{e}_B)$ is acyclic 
if and only if $\alpha(\theta_B)t^{\sigma(\theta_B)} \neq 1$.
By definition of $\alpha \in H^1(N; \ZkZ k)$ we have $\alpha(\theta_B) = \exp{2\pi i/k} \in \IC$ and so 
$A\otimes \IZ[\pi_1(B)] \otimes C(B)$ is acyclic.
We obtain 
\begin{align*}
	\tau(B, A \otimes \IZ[\pi_1(B)]) 
	&= \frac{(1-\alpha(\theta_B)t^{\sigma(\theta_B)})^{\rk H_1(\Sigma; \IZ)}}
		{(1-\alpha(\theta_B)t^{\sigma(\theta_B)})^2}\\
	&= (1-\alpha(\theta_B)t^{\sigma(\theta_B)})^{-\chi(\Sigma)} \in W(A \otimes \IZ[\pi_1(B)])
\end{align*}
Consequently, we obtain the equality of the conclusion in the quotient $W(A)$.
\end{enumerate}
\end{proof}
\begin{prop}\label{TorsionProduct}
Let $\sigma \in H^1(N;\IZ)$ be a cohomology class.
\begin{enumerate}
\item The chain complex $\IC_{\alpha \otimes \sigma}(t) \otimes_{\IZ[\pi_1]} C(N, \mathfrak{e})$ is acyclic. 
\item The torsion equals
\begin{align*}
	\tau(N, \IC(t)_{\alpha\otimes \sigma}) = 
	\prod_{B \in \mathcal{B}_N}\left(1 -\alpha(\theta_B) t^{\sigma(\theta_B)} \right)^{-\chi(\Sigma_B)}.
\end{align*}
\end{enumerate}
\end{prop}
\begin{proof}
Let $\sigma \in H^1(M;\IZ)$ be arbitrary.
\begin{enumerate}
\item We have seen that for each block $B$ the complex 
$\IC(t)_{\alpha\otimes \sigma} \otimes_{\IZ[\pi_1]} C(B \subset N, \mathfrak{e}|_B)$
is acyclic. 
A similar consideration shows that $\IC(t)_{\alpha\otimes \sigma}\otimes_{\IZ[\pi_1]} C(T \subset N, \mathfrak{e}|_T)$ 
is acyclic, see \cite[Lemma 11.11]{Turaev01}. 
By Lemma \ref{TorsionSplit} the chain complex $A \otimes_{\IZ[\pi_1]} C(N, \mathfrak{e})$ is acyclic as well.
\item In \cite[Lemma 11.11]{Turaev01} the torsion of a torus is calculated, i.e.
\begin{align*}
	\prod_{T \in \mathcal{T}}\tau(T \subset N,\IC(t)_{\alpha\otimes \sigma}) = 1.
\end{align*}
By Lemma \ref{TorsionSplit} we get
\begin{align*}
	\tau(N,\IC(t)_{\alpha\otimes \sigma}) 
	\prod_{T \in \mathcal{T}_N}\tau(T \subset N,\IC(t)_{\alpha\otimes \sigma}) 
	&= \prod_{B \in \mathcal{B}_N} \tau(B \subset N, \IC(t)_{\alpha\otimes \sigma})
\end{align*}
and consequently deduce
\begin{align*}
	\tau(N, \IC(t)_{\alpha\otimes \sigma}) = 
	\prod_{B \in \mathcal{B}_N}\left(1 -\alpha(\theta_B) t^{\sigma(\theta_B)} \right)^{-\chi(\Sigma_B)}.
\end{align*}
\end{enumerate}
\end{proof}

\subsection{The Thurston and the torsion norm}
In this section, first we give the definition of the torsion norm, and
then we prove that every graph manifold~$M$ with $b_2(M) \geq 2$
admits a finite cover such that the Thurston norm 
and the torsion norm agree.

In this section, we assume that the $3$-manifold~$M$ has $b_2(M) \geq 2$.
Turaev \cite{Turaev76} introduced the maximal abelian torsion. 
In the case of $b_2(M) \geq 2$, it is an element $\Delta_M$ of the group ring $\IZ[H_1(M;\IZ)]$, which
is an invariant of the pair $(M, \mathfrak{e})$.
Our main use of this invariant is that it encodes the Seiberg-Witten basic
classes of $M$. Its construction can be found in \cite[Section 13]{Turaev01}.
The proposition below illustrates the close relation between the maximal abelian torsion
and the twisted Reidemeister torsion.

Recall the ring homomorphism~$\rho_{\alpha\otimes \sigma}$ from Construction~\ref{Constr:Twist}.
\begin{align*}
\Map { \rho_{\alpha\otimes \sigma} } {\IZ[\pi_1]&} {\IC(t)}\\
g &\mapsto \alpha(g) t^{\sigma(g)}
\end{align*}

\begin{prop}\label{Prop:MAbelian}
Let $\mathfrak{e}$ be a fundamental family for $M$. 
Denote by $\Delta_M \in \IZ[H_1(M;\IZ)]$ the maximal abelian torsion with
respect to $\mathfrak{e}$. Let $\alpha \in \Hom(\pi_1, \ZkZ k)$ be a character
and $\sigma \in H^1(M; \IZ)$ be a cohomology class.
Assume the chain complex 
$\IC(t)_{\alpha\otimes \sigma} \otimes_{\IZ[\pi_1]} C(M, \mathfrak{e})$ is acyclic.
Then the following equality holds:
\begin{align*}
\rho_{\alpha\otimes \sigma} (\Delta_M) = \tau(M,\IC(t)_{\alpha\otimes \sigma}) \in \W {\IC(t)_{\alpha\otimes \sigma}}.
\end{align*}
\end{prop}
\begin{proof}
See \cite[Theorem 13.3]{Turaev01}.
\end{proof}

With the help of the maximal abelian torsion $\Delta \in \IZ[H_1(M;\IZ)]$ 
we define a norm on $H_1(M; \IZ)$ below. A variant, the Alexander norm, of this seminorm has been
investigated and compared with the Thurston norm in \cite{McMullen02}.
\begin{defn}\label{defn:TorsionNorm}
Expand the maximal abelian torsion in the homology variables
$\Delta_M = \sum_h a_h h \in \IZ[H_1(M;\IZ)]$.
Define the \emph{torsion norm} $\TorN{.}$ on $H^1(M; \IZ)$ by
\begin{align*}
	\TorN \sigma := \max \left\{ \sigma (h) - \sigma( h' ) : h,h' \in H_1(M; \IZ) \text{ with } a_h, a_{h'} \neq 0  \right\}
\end{align*}
\end{defn}
\begin{rem}
\begin{enumerate}
\item If one switches from one fundamental family to another, 
then the maximal abelian torsion 
gets multiplied by $\pm h$ for a suitable $h \in H_1(N; \IZ)$, see \cite[Section II.13 p. 65]{Turaev01}.
Although the maximal abelian torsion does depend on the choice of fundamental 
family, the torsion norm does not.
\item Above we use the convention that the above maximum of the empty set
is 0.
\end{enumerate}
\end{rem}

It is well known that the twisted Reidemeister torsion gives a lower bound on the Thurston
norm, see \cite[Theorem 1.1]{FriedlKim06}, and \cite[Theorem 1.2]{Friedl13} for a purely torsion theoretic argument.
\begin{prop}\label{Comparison}
Let $\sigma \in H^1(M; \IZ)$ be a cohomology class. Denote by $\Delta_M$ the maximal abelian torsion of $M$.
Then
\begin{align*}
	\Thur \sigma \geq \TorN \sigma \geq \width \rho_{\alpha \otimes \sigma}(\Delta)
\end{align*}
\end{prop}
\begin{proof}
See \cite[Theorem IV.2.2]{Turaev02} for a prove of the first inequality.
The second inequality can be deduced as follows. Writing $\Delta = \sum_h a_h h$ 
	we obtain the inequality
	\begin{align*}
		\width \rho_{\alpha \otimes \sigma} (\Delta) &= \width \sum_h a_h \alpha(h) t^{\sigma(h)}\\
		&=\width \sum_k \left(\sum_{\sigma(h) = k} a_h \alpha(h) \right) t^k\\
		&\leq \max \left\{\sigma (h) - \sigma( h' ) : h \in H_1(M; \IZ) \text{ with } a_h, a_{h'} \neq 0  \right\}.
	\end{align*}
\end{proof}
\begin{prop}\label{GraphThurstonTorsion}
Let $N$ be virtually of composite type. Then there exists a
finite cover $\pi \colon M \rightarrow N$ 
such that Thurston norm agrees with the torsion norm on $N$.
\end{prop}
\begin{proof}
Pick a cover $\pi \colon M \rightarrow N$ such that $M$ fits
in the setting of Section~\ref{CompositeFolds}. We denote the set
of blocks by $\mathcal{B}_M$ and the homology class of the Seifert fibre in each block~$B \cong \Sigma_B \times S^1$
by $\theta_B$. Let $\sigma \in H^1(M;\IZ)$ be arbitrary. 
We will show $\Thur \sigma = \TorN \sigma$.
Again we fix a cohomology class $\alpha \in H^1(N; \ZkZ k)$
for a natural number $k \geq 2$ which evaluates on every Seifert fibre in every block $B$ 
to $\alpha(\theta_B) = 1$.
We consider the module $\IC(t)_{\alpha \otimes \sigma}$ from Construction~\ref{Constr:Twist}.

From Proposition~\ref{TorsionProduct} 
we know that
\begin{align*}
	\tau(M, \IC(t)_{\alpha\otimes \sigma}) = 
	\prod_{B \in \mathcal{B}_M}\left(1 -\alpha(\theta_B) t^{\sigma(\theta_B)} \right)^{-\chi(\Sigma_B)}.
\end{align*}
Applying the function $\width$ we obtain
\begin{align*}
\width \tau(M, \IC(t)_{\alpha\otimes \sigma}) 
&= \sum_{B \in \mathcal{B}_M} \width \left(1 - \alpha(\theta_B) t^{\sigma(\theta_B)}\right)^{-\chi(\Sigma_B)}\\
&= \sum_{B \in \mathcal{B}_M} -\chi(\Sigma_B) \width \left(1 - \alpha(\theta_B) t^{\sigma(\theta_B)}\right).
\end{align*}
We have the equality: $\width \left( 1 - \alpha([\theta_B]) t^{\sigma([\theta_B])} \right) = | \sigma([\theta_B]) |$. 
Thus by Proposition \ref{ThurstonSplit} and \ref{ThurstonProduct}
we obtain the equality
\begin{align*}
\width \tau(M, \IC(t)_{\alpha\otimes \sigma}) = \sum_{B \in \mathcal{B}} \Thur {i_B^* \sigma}
= \Thur \sigma
\end{align*}
With Propositions \ref{Prop:MAbelian} and \ref{Comparison} and the equality above, we obtain
\begin{align*}
\Thur \sigma = \TorN \sigma.
\end{align*}
\end{proof}
\begin{rem}\label{MaxAbNonVan}
Proposition \ref{TorsionProduct}
combined with Proposition~\ref{Prop:MAbelian} and \ref{Comparison} shows that the
maximal abelian torsion of $N$ has to be non-zero.
\end{rem}

\begin{thm}\label{AlexanderThurston}
Let $N$ be a graph manifold with $b_1(N) \geq 2$. 
Then there is a finite cover
$p\colon M \rightarrow N$ such that the torsion norm and
the Thurston norm coincide on $M$.
\end{thm}
\begin{proof}
We deal with the types of Definition \ref{Defn:GraphTypes} one by one:
\begin{description}
\item[Spherical type] Spherical manifolds cannot have $b_1(N) \geq 2$.
\item[Circle type] In this case $N$ is covered by a non-trivial circle bundle $E$. 
Thus by \cite[Example 7.3]{McMullen02} the Thurston vanishes. 
As a consequence also the torsion norm has to vanish.
\item[Torus type] 
Again we will prove that $\Thur .$ vanishes identically.
Recall that a mapping torus $\MapT(T^2, \phi)$ fibres over $S^1$. 
We denote the corresponding map
by $\Map q {\MapT(T^2, \phi)} {S^1}$. Fix a fibre $T^2 \subset \MapT(T^2, \phi)$
and denote its inclusion by $\Map j {T^2} {\MapT(T^2, \phi)}$.
The sequence below is exact.
\begin{align*}
	\ldots 	\overset{\partial}{\rightarrow} H_k(T^2;\IQ) 
	\overset{\phi_* - \id}{\longrightarrow} H_k(T^2;\IQ) 
	\overset{j_*}{\rightarrow} H_k(\MapT(T^2, \phi);\IQ) 
	\overset{\partial}{\rightarrow} H_{k-1}(T^2;\IQ) \rightarrow \ldots
\end{align*}
It can be derived from the Mayer-Vietoris sequence of a suitable cover of the mapping torus.
For every element $\alpha \in H_1(T^2;\IZ)$ we can find an embedded loop 
$\gamma \subset T^2$ with $[\gamma] = \alpha$. If $\alpha$ is in the kernel
of $\phi_* - \id$, then $\phi(\gamma)$ is isotopic to $\gamma$ itself.
Because of this we can find an horizontal embedded torus  $T_\gamma \subset \MapT(T^2, \phi)$ such 
that $\partial [T_\gamma] = \alpha$.

Adding a multiple of the fundamental class of a suitable torus and using the reverse 
triangle inequality, this shows that for every class 
$\sigma \in H_2\left( \MapT(T^2, \id); \IQ\right)$
we can find a class $\sigma'$ with $\Thur \sigma = \Thur {\sigma'}$ and
$\partial \sigma' = 0$. This implies that $\sigma'$ is a multiple of the fundamental class
of the torus fibre. Therefore the Thurston norm $\Thur {\sigma'} = 0$ vanishes.
\item[Composite type] This is Proposition~\ref{GraphThurstonTorsion}.
\end{description}
\end{proof}
\section{Embedded surfaces in circle bundles}
We define a complexity for elements of $H_2(W;\IZ)$ in a $4$-manifold $W$. 
Similar to the Thurston norm, it measures the minimal 
complexity of an embedded surfaces representing the given class.
\begin{defn}\label{Def:Complexity}
Let $W$ be a $4$-manifold. The \emph{complexity} of a class $\sigma \in H_2(W;\IZ)$
is
\begin{align*}
\Comp(\sigma) := \min \{ \chi_-(\Sigma) : 
	\Sigma \text{ a properly embedded surface with } \sigma = [\Sigma]\},
\end{align*}
where $\chi_-(\Sigma)$ is defined in Definition \ref{SurfacesComplexity}.
\end{defn}

We use the theory of Seiberg-Witten invariants and the adjunction inequality
to obtain estimates for the complexity of embedded surfaces. 
Denote the collection of $\Spinc$-structures on a $4$-manifold $W$ by $\Spinc(W)$.
We will only apply the Seiberg-Witten theory in the case $b_2^+(W) \geq 2$. 
The Seiberg-Witten invariant, introduced by Witten \cite{Witten94}, is a function
$\Map {SW_W} {\Spinc(W)} \IZ$. It is non-zero only on finitely many $\Spinc$-structures.
Details on the construction can be found in \cite{Morgan96}. 
Associated with each element $\xi \in \Spinc(W)$ is the \emph{determinant line bundle},
see \cite[Section 3.1]{Morgan96}. The Chern class of this line bundle
is denoted by $c_1(\xi) \in H^2(W;\IZ)$.
The collection of Seiberg-Witten \emph{basic classes} is
\[ \Bas W := \{ c_1(\xi) : \xi \in \Spinc(W) \text{ with }\SW_W(\xi) \neq 0\} \subset H^2(W;\IZ). \]

We also need to define basic classes for a closed $3$-manifold $N$ with $b_1(N) \geq 2$.
For such a $3$-manifold we have $b_2^+(N \times S^1) = b_1(N) \geq 2$.
The projection $\Map p {N \times S^1} N$ gives rise to a map 
\begin{align*}
\Map {p^*} {\Spinc(N)} {\Spinc(N \times S^1)}.
\end{align*}
We describe this map in greater detail in the context of circle bundles in Section
\ref{SWforCirclebundles}. 
As above, we associate to each element $\xi \in \Spinc(N)$, a class  $c_1(\xi)$, 
the Chern class of the determinant line bundle.
The Chern classes $c_1(\xi)$ and
$c_1(p^* \xi)$ are related by the equation $p^* c_1(\xi) = c_1( p^* \xi)$.
We define the collection of basic classes on $N$ to be
\begin{align*}
\Bas N := \{ c_1(\xi) : \xi \in \Spinc(N) \text{ with } \SW_{N \times S^1}( p^* \xi) \neq 0\}.
\end{align*}

For our purposes, it will be enough to define the Seiberg-Witten invariant 
of a $\Spinc$-structure~$\xi$ of $3$-manifold~$N$ 
in terms of the invariant of the $4$-manifold $N\times S^1$ by 
\[ \SW_N(\xi) := \SW_{N\times S^1}(p^*\xi). \]
It is possible to express the $3$-dimensional Seiberg-Witten invariants~$\SW_N$ for a $3$-manifold
purely in terms of the moduli space of Seiberg-Witten monopoles over $N$.
See \cite[Section 2]{MengTaubes96} for the equations and the construction
of the moduli space. These two approaches give the same invariant \cite[Proposition 7]{Baldridge01}.
We refer the reader to Nicolaescu~\cite[Chapter 4]{Nico03} for an  introduction also 
explaining the connection with Reidemeister torsion,

The next definition is a rather technical condition, which is fulfilled for a large
class of irreducible $3$-manifolds. With this condition in hand, we can 
state the main theorem of this article. 

Before the definition, we quickly recall the notion of transfer;
see \cite[Definition 11.2]{Bredon93} for more details. Let $\Map f M N$
be a map between two closed oriented manifold of the same dimension.
The \emph{transfer map}~$\Map {f^!} {H_k (N;\IZ)} {H_k(M; \IZ)}$ is
\[ f^! = \PD_M \circ f^* \circ \PD_N^{-1}, \]
where $f^*$ denotes the induced map in cohomology. 
As suggested by the notation, associating $f^!$ to $f$ is contravariant.

\begin{defn}
A $3$-manifold $N$ has \emph{enough basic classes} if the following holds: 
for every finite cover $\Map g M N$ and every class $\sigma \in H_2(N;\IZ)$ 
there is a finite cover $\Map h P M$ with $b_1(P) \geq 3$ and a basic class $s \in \Bas(P)$ 
such that
\begin{align*}
	\Thur {(g\circ h)^!\sigma} = \langle s, (g\circ h)^!\sigma \rangle.
\end{align*}
\end{defn}

\begin{thm}\label{FinalEstimate}
Let $N$ be an irreducible and closed $3$-manifold.
Let $\Map p W N$ be an oriented circle bundle over $N$. Suppose $N$ has
enough basic classes.
Then each class $\sigma \in H_2(W;\IZ)$ satisfies the inequality
\begin{align*}
	\Comp(\sigma) &\geq | \sigma \cdot \sigma | + \Thur{p_* \sigma }.
\end{align*}
\end{thm}
\begin{proof}
This is a combination of Lemma~\ref{TorsionCase} and Lemma~\ref{NonTorsionCase} below. 
The term \emph{non-degenerate} is defined in Definition~\ref{defn:NonDegen}.
\end{proof}

We discuss which graph manifolds have enough basic classes. Eventually,
we will obtain the theorem stated in the introduction in the form of Theorem~\ref{Thm:AResult}. 

In the case of an aspherical $3$-manifold with virtually RFRS fundamental group, Friedl-Vidussi~\cite{Friedl14}
made essential use that these $3$-manifolds virtually fibre by Agol's virtual fibred theorem~\cite{Agol08}.
This fails dramatically
in the realm of graph manifolds. As already mentioned, there are graph manifolds which do not
admit a single finite cover which fibres. In the case of graph manifolds we will therefore
rely on the description of basic classes in terms of the maximal abelian torsion:
Building on \cite{MengTaubes96} the equality between the Seiberg-Witten invariant
and the torsion function was described in \cite[Theorem 1]{Turaev98}.
The following theorem is a direct consequence of this equality.
We again adopt the convention here that the maximum over the empty set is $0$.
\begin{thm}[Turaev]\label{AlexanderBasicClass}
Let $N$ be a closed $3$-manifold with $b_1(N) \geq 2$.
\begin{enumerate}
\item If the maximal abelian torsion is non-zero then $\Bas N$ is non-empty.
\item Each class $\sigma \in H_2(N; \IZ)$ satisfies the equation
\begin{align*}
\TorN {\PD \sigma} = \max_{s \in \Bas N} \langle s, \sigma \rangle.
\end{align*}
\end{enumerate}
\end{thm}
\begin{proof}
See \cite[Section 21]{Turaev01} and \cite[Section IX.1.2]{Turaev02}.
\end{proof}

The next lemma shows that the class of graph manifolds with enough
basic classes is rather large.
\begin{lem}\label{lem:GraphCase}
Let $N$ be a graph manifold which is not Seifert fibred and not covered by a torus bundle.
Then $N$ has enough basic classes.
\end{lem}
\begin{proof}
We are given a finite cover $\Map g M N$ and a class $\sigma \in H_2(N;\IZ)$. 
In fact we will see that we can find a suitable cover $\Map h P M$ which does not depend on the choice
of $\sigma$.
As $N$ is not Seifert fibred and not covered by a torus bundle, 
the fundamental group $\pi_1(N)$ is not solvable \cite[Theorem 1.20]{Aschenbrenner12}.
Using the fact that it has to contain at least one incompressible torus, 
there is a finite cover
$\widetilde M \rightarrow M$ with $b_1(\widetilde M) \geq 3$, see \cite[Diagram 1]{Aschenbrenner12}.
The $3$-manifold~$M$ is again not covered by Seifert fibred manifold nor covered by torus bundle. 
Thus $M$ is neither virtually of circle type, spherical type nor torus type.
Using Lemma~\ref{Lem:GraphTypes} we deduce that $M$ has to be virtually of composite type.
Pick a cover $P \rightarrow \widetilde M$ such that $P$ is of composite type and is fulfilling
the assumptions of Section \ref{CompositeFolds}.
The induced cover $\Map h P M$ inherits the condition
$b_1(P) \geq 3$ from $\widetilde M$. 
We are left to prove that there is an $s \in \Bas P$ such that
\begin{align*}
	\Thur {(g\circ h)^!\sigma} = \langle s, (g\circ h)^!\sigma \rangle.
\end{align*}
If the Thurston norm of $P$ 
does not vanish, then this can be seen by Theorem \ref{AlexanderBasicClass} 
and Proposition \ref{TorsionProduct}.
If the Thurston norm vanishes we still have to prove that $\Bas P$ is non-empty.
Again this holds by Theorem \ref{AlexanderBasicClass} and Remark \ref{MaxAbNonVan}.
\end{proof}

\begin{thm}\label{Thm:AResult}
Let $N$ be an irreducible and closed $3$-manifold.
Assume that $N$ is not a Seifert fibred space and not covered by a torus bundle.
Let $\Map p W N$ be an oriented circle bundle. Then each class
$\sigma \in H_2(W;\IZ)$ satisfies the inequality 
\begin{align*}
\Comp(\sigma) \geq | \sigma \cdot \sigma | + \Thur {p_* \sigma}.
\end{align*}
\end{thm}
\begin{proof}
A closed graph manifold has enough basic classes by Lemma~\ref{lem:GraphCase}.

If $N$ is not a closed graph manifold, then $\pi_1(N)$ is virtually
special by work of Kahn-Markovic~\cite{Kahn12},
Wise~\cite{Wise11}, Bergeron-Wise~\cite{Wise12}, 
Przytycki-Wise~\cite{Przy12} and Agol~\cite{Agol13}. 
By Agol~\cite[Corollary 2.3]{Agol08} and Haglund-Wise~\cite{Haglund08}
the group $\pi_1(N)$ is virtually RFRS.
A closed irreducible and aspherical $3$-manifold with virtually
RFRS fundamental group has enough basic classes \cite[Corollary 3.4]{Friedl14}. 
This is a consequence of Agol's virtually fibred theorem \cite[Theorem 5.1]{Agol08}
in combination with a result due to Taubes \cite[Main Theorem]{Taubes94} calculating the Seiberg-Witten invariant
associated with a symplectic structure.

Either way, we conclude that our $3$-manifold has enough basic classes, and
in the light of Theorem~\ref{FinalEstimate} we obtain the inequality.
\end{proof}
We continue with the discussion of the proof of Theorem~\ref{FinalEstimate}.
A central tool for estimating the complexity of embedded surfaces
is the adjunction inequality. For positive self-intersection it has been obtained by 
Kronheimer and Mrowka \cite[Section 6]{KM94}.
Later Ozsváth and Szabó \cite[Corollary 1.7]{OS00} proved it for negative self-intersection 
in $4$-manifolds of simple type.
We use the following variant, which is specific to circle bundles but still relies on 
the adjunction inequality for positive self-intersection.
\begin{lem}\label{AdjunctionInequality}
	Let $W$ be a circle bundle over an irreducible $3$-manifold $N$ 
	with $b_1(N) \geq 3$. Then for
	all $\sigma \in H_2(W;\IZ)$ and all basic classes $s \in \Bas W$ 
	the inequality
	\begin{align*}
		\Comp(\sigma) \geq | \sigma \cdot \sigma | 
				+ \langle s, \sigma \rangle
	\end{align*}
	holds.
\end{lem}
\begin{proof}
	See \cite[Theorem 3.1]{Friedl14}.
\end{proof}
We dissect the proof of Theorem \ref{FinalEstimate} into multiple parts.
First we study how the complexity function $\Comp$ changes in finite covers.
In the second step we explain the proof of the case that the Euler class of the circle
bundle is torsion. 
Key here is that for a trivial
circle bundle the pull-back of basic class is still a basic class.
As we will see, this also holds in the non-torsion case for all but finitely many Euler classes.
A change of the circle bundle by a modification away from a minimal surface 
gives a way to work around these Euler classes. But to be able to modify 
the circle bundle away from a surface we might have to increase its complexity.
This again can be countered by taking finite covers.
\subsection{Finite covers and embedded surfaces}
We study how the various quantities in the inequality change in finite covers.
Let $\Map p W N$ be a circle bundle over a $3$-manifold $N$. In all steps of the proof 
of Theorem \ref{FinalEstimate} we only
need to consider finite covers $W$ which are induced by a finite cover of $N$.
Let $\Map{g}{M}{N}$ be a finite cover. 
We pull back the circle bundle $\Map p W N$ to a circle bundle $\Map {\widetilde p} {g^*W} M$. 
The induced map $\Map {\widetilde g} {g^*W} W$ is a finite cover.
The situation is described in the diagram
\begin{center}
	\begin{tikzpicture}
		\matrix (m) [matrix of math nodes, row sep=3em, column sep=4em,
					text height=1.5ex, text depth=0.25ex]
  		{
		     g^* W & W \\
		     M & N\\
		};
		\path[->] (m-1-1) edge node [left] {$\widetilde p$} (m-2-1);
		\path[->] (m-1-1) edge node [above] {$\widetilde g$} (m-1-2);
		\path[->] (m-1-2) edge node [right] {$p$} (m-2-2);
		\path[->] (m-2-1) edge node [below] {$g$} (m-2-2);
	\end{tikzpicture}.
\end{center}
The following definition is rather auxiliary. In the end our goal is to prove
that all embedded surfaces are non-degenerate.
\begin{defn}\label{defn:NonDegen}
Let $\Map p W N$ be a circle bundle.
\begin{enumerate}
\item An embedded surface $\Sigma \subset W$ is \emph{non-degenerate}
if
\begin{align*}
	\chi_-(\Sigma) \geq |\sigma \cdot \sigma | + \Thur {p_* \sigma}
\end{align*}
holds for its fundamental class $\sigma \in H_2(W;\IZ)$.
\item A class $\sigma \in H_2(W;\IZ)$ is said to be \emph{non-degenerate}
if 
\begin{align*}
	\Comp {(\sigma)} \geq |\sigma \cdot \sigma | + \Thur {p_* \sigma}
\end{align*}
holds.
\end{enumerate}
\end{defn}

The next lemma explains how non-degenerateness is inherited from 
the finite covers we have described above.
\begin{lem}\label{FiniteCovers}
In the situation above, let $\Sigma \subset W$ be an embedded surface
and $\sigma \in H_2(W;\IZ)$ its fundamental class.
\begin{enumerate}
\item The equality $\widetilde p_* \widetilde g^! = g^! p_*$ holds.
\item If the surface $\widetilde g^{-1} (\Sigma)$ is non-degenerate, then
so is the surface $\Sigma$.
\item If the transfer $\widetilde g^! \sigma$ is non-degenerate, then
so is the class $\sigma$.
\end{enumerate}
\end{lem}
\begin{proof}
We abbreviate the transfer of $\sigma$ under the map $\widetilde g$ 
to $\widetilde \sigma := \widetilde g^! \sigma$.
\begin{enumerate}
	\item 	Note that $g$ and $\widetilde g$ are finite covers
		of the same degree and $g^!$ and $\widetilde g^!$ agree
		with the transfer maps of the corresponding finite covers.
		With this in mind the equality can be checked already on
		singular chain level.

	\item 	Because $\widetilde g$ is a covering map of degree $\deg g$,
		we have the equality
		\begin{align*}
			\chi_-(\widetilde g^{-1} (\Sigma)) = \deg g \cdot \chi_-(\Sigma).
		\end{align*}
		Also the self-intersection is multiplicative, i.e.
		\begin{align*}
		\deg{g}\; | \sigma \cdot \sigma | 
		= |\widetilde \sigma \cdot \widetilde \sigma|.
		\end{align*}
		We are left to show the equality
		\begin{align*}
			\Thur{\widetilde p_* \widetilde \sigma} 
			=\Thur{g^! p_*\sigma} 
			= \deg g \Thur{p_*\sigma}.
		\end{align*}
		We have already proved the first equality. The second
		equality follows from the fact that the Thurston
		norm is multiplicative under finite covers. 
		This was proven by Gabai \cite[Corollary 6.13]{Gabai83}.
	\item We have to show the inequality $\deg {g}\cdot \Comp(\sigma) 
		\geq \Comp(\widetilde \sigma)$.
		Let $\Sigma$ be a
		surface representing $\sigma$. The surface $\widetilde g^{-1}(\Sigma)$
		represents $\widetilde \sigma$. It has
		complexity $\chi_-( \widetilde g^{-1}(\Sigma) ) 
			= \deg g \cdot \chi_- (\Sigma)$.
		Therefore we obtain
		\begin{align*}
			\deg g\; \Comp(\sigma) &= \inf_{\Sigma} \chi_-( \widetilde g^{-1}(\Sigma) ) 
			\geq \Comp(\widetilde \sigma).
		\end{align*}
	\end{enumerate}
\end{proof}
\subsection{Basic classes of circle bundles}
\label{SWforCirclebundles}
For the rest of the article let $\Map p W N$ be a circle bundle over an connected and closed $3$-manifold $N$
with $b_1(N) \geq 3$. This implies $b_2^+(W) \geq b_1(N) - 1 = 2$, see \cite[Section 2]{Friedl14}.
Additionally suppose $N$ has enough basic classes.
Recall that by convention $N$ is irreducible and $W$
has oriented fibres. We denote the Euler class of $W$ by $e \in H^2(N;\IZ)$.
A connection on $W$ determines a splitting of the tangent bundle
\begin{align*}
TW = p^* TN \oplus \underline{\IR}.
\end{align*}
The pull-back of a $\Spinc$-structure on $N$ along $p$ gives rise to
a $\Spinc(3)$-structure on $W$, i.e. a $\Spinc(3)$-principal bundle $P$ over $W$
with a vector bundle isomorphism 
$\Spinc(3) \times_{\SO(4)} \IR^4 \cong  p^* TN \oplus \underline{\IR}$.
Prolonging along $\Spinc(3) \subset \Spinc(4)$, we obtain a $\Spinc$-structure $p^* \xi$ 
on $W$. We have the equality $c_1(p^* \xi) = p^* c_1(\xi)$.
Baldridge expressed the Seiberg-Witten invariants on $W$ 
in terms of the Seiberg-Witten invariants on $N$
and the Euler class. This is described in the next theorem.
\begin{thm}[Baldridge]\label{BaldrigeFormula}
Let $\xi \in \Spinc(N)$ be a $\Spinc$-structure.
If $e$ is non-torsion, then the equality
\begin{align*}
\SW_W(p^* \xi) = \sum_{k \in \IZ} \SW_N(\xi + k\cdot e)
\end{align*}
holds.
\end{thm}
\begin{proof}
\cite[Theorem 1]{Baldridge01}
\end{proof}

We proceed by proving Theorem \ref{FinalEstimate} in the torsion case.
\begin{lem}\label{TorsionCase}
	If the Euler class $e$ is torsion, then every class
	$\sigma \in H_2(W;\IZ)$ is non-degenerate.
\end{lem}
\begin{proof}
	Let $\sigma$ be a class in $H_2(W;\IZ)$.
	The Euler class $e$ of the circle bundle is torsion.
	This implies that there is a finite cover $\Map{g}{M}{N}$ such that 
	$\Map {\widetilde p} {g^* W} M$ is the trivial circle bundle, see \cite[Proposition 3]{Bowden09}.
	We know that $N$ has enough basic classes. As a consequence, we may assume that 
	there is a basic class $s \in \Bas M$ with
	$\Thur {g^! p_* \sigma} = \langle s, g^! p_* \sigma \rangle$.
	
	Denote $\widetilde g^! \sigma \in H_2(g^*W;\IZ)$ by $\widetilde \sigma$. 
	Using the equality $g^! p_* = \widetilde p_* \widetilde g^!$ stated in Lemma \ref{FiniteCovers},
	we obtain the equation $\Thur {\widetilde p_* \widetilde \sigma} = \langle s, \widetilde p_*\widetilde \sigma \rangle$.
	Because the circle bundle is trivial, the class $s$ is a basic class if and only if
	$\widetilde p ^* s$ is a basic class.
	The adjunction inequality gives
	\begin{align*}
		\Comp(\widetilde \sigma) 
		\geq | \widetilde \sigma \cdot \widetilde \sigma | 
			+ \langle \widetilde p^*s, \widetilde \sigma \rangle
		= | \widetilde \sigma \cdot \widetilde \sigma | 
			+ \langle s, \widetilde p_*\widetilde \sigma \rangle
		= | \widetilde \sigma \cdot \widetilde \sigma | 
			+ \Thur {\widetilde p_*\widetilde \sigma}.
	\end{align*}
	Thus $\widetilde \sigma \in H_2(g^* W;\IZ)$ 
	is non-degenerate and so is $\sigma \in H_2(W;\IZ)$ 
	by Lemma \ref{FiniteCovers}. 
\end{proof}
If the Euler class $e$ is non-torsion, there is the following problem: 
Even if we know that $s \in \Bas N$ is a basic class, it could happen that all summands
in Theorem \ref{BaldrigeFormula} cancel.
The next lemma shows that the Euler class is sensitive to local changes near an embedded loop
while preserving the rest of the circle bundle.
The key idea which allows to improve the result of Friedl and Vidussi~\cite{Friedl14} is
to change the Euler class away from a surface in a way that 
all summands in Theorem \ref{BaldrigeFormula} but $\SW_N(\xi)$ become zero.
\begin{lem}\label{EulerClone}
Let $\gamma \subset N$ be an embedded loop and $k \in \IZ$ an integer. 
There is a circle bundle 
$\Map {p_k} {W_k} N$ with Euler class $e(W_k) = e + k \PD([\gamma])$, and
an isomorphism of principal $S^1$-bundles $W_k |_{N\setminus {\gamma}} \cong W|_{N\setminus {\gamma}}$.
\end{lem}
\begin{proof}
Pick an oriented circle bundle $W_k$ with Euler class
$e(W_k) = e + k \PD([\gamma])$.
We have to prove that the pull-backs of $W_k$ and $W$ along the
inclusion $\Map i {N\setminus \gamma} {N}$ give isomorphic principal $S^1$-bundles.
As the Euler class classifies such isomorphism classes, it will be enough to prove 
$i^*e(W_k) = i^* e$. This is equivalent to $i^* \PD([\gamma]) = 0$. We pick an
open tubular neighbourhood $\nu(\gamma)$ of $\gamma$ and denote its boundary by
$\partial \nu(\gamma)$. As $N \setminus \gamma$ deformation retracts to
$N \setminus \nu(\gamma)$, it is sufficient to prove the equality
$i_\nu^* \PD([\gamma]) = 0$ with $\Map {i_\nu} {N \setminus \nu(\gamma)} N$ being
the inclusion.
This equality follows from the commutative diagram
\begin{center}
	\begin{tikzpicture}
		\matrix (m) [matrix of math nodes, row sep=3em, column sep=4em,
					text height=1.5ex, text depth=0.25ex]
  		{
		     H^2(N; \IZ) & H^2(N \setminus \nu(\gamma); \IZ) \\
		     H_1(N;\IZ) & H_1(N \setminus \nu(\gamma), \partial \nu(\gamma);\IZ)\\
		};
		\path[->] (m-1-1) edge node [left] {$\PD$} (m-2-1);
		\path[->] (m-1-1) edge node [above] {$i_\nu^*$} (m-1-2);
		\path[->] (m-1-2) edge node [right] {$\PD$} (m-2-2);
		\path[->] (m-2-1) edge (m-2-2);
	\end{tikzpicture}
\end{center}
which is a consequence of Poincaré-Lefschetz duality and excision, see \cite[Corollary VI.8.4]{Bredon93}.
\end{proof}

\begin{lem}\label{DisjointLoop}
Let $\Sigma \subset W$ be 
an oriented surface embedded in $W$. 
If there is an embedded
loop $\gamma$ such that its class $[\gamma] \in H_1(N;\IZ)$ is non-torsion
and $p^{-1}(\gamma)$ is disjoint from $\Sigma$,
then the surface $\Sigma$ is non-degenerate.
\end{lem}
\begin{proof}
We abbreviate the fundamental class of $\Sigma$ to $\sigma \in H_2(W;\IZ)$.
\begin{description}
\item[Special case] We first consider the case where a basic class $s \in \Bas(N)$ with
$\langle s, p_*\sigma \rangle = \Thur {p_*\sigma}$ exists.
Denote the Euler class of $W$ by $e(W)$.
Let $\gamma$ be a loop fulfilling the assumption. Pick circle bundles $W_k$ 
as given by Lemma \ref{EulerClone} and the corresponding isomorphism 
of principal $S^1$-bundles $W_k |_{N\setminus {\gamma}} \cong W|_{N\setminus {\gamma}}$.
Via the induced diffeomorphism of the total spaces we inherit an embedding
\begin{align*}
	\Sigma \subset W|_{N\setminus {\gamma}} = W_k|_{N\setminus {\gamma}} \subset W_k
\end{align*}
The fundamental class of $\Sigma$ in $W_k$ via this embedding is denoted by $\sigma_k \in H_2(W_k;\IZ)$.
Note that the self-intersection of $\Sigma$ only depends on a neighbourhood of $\Sigma$. Thus
the self-intersections $\sigma \cdot_W \sigma = \sigma_k \cdot_{W_k} \sigma_k$ agree.
Furthermore, we have $p_* \sigma = {p_k}_* \sigma_k$. Therefore it is enough to prove
that $\Sigma$ is non-degenerate in $W_k$ for a single $k \in \IZ$.

Abbreviate the Euler class $e(W_k) := e(W) + k \PD([\gamma])$ of $W_k$.
Because $[\gamma]$ is non-torsion
and $\SW_N$ has finite support, we can choose $k$ large enough such that
the following holds: For each $\xi \in \Spinc(N)$ there is at most one
$l \in \IZ$ such that $\SW_N(\xi + l e(W_k)) \neq 0$. Pick such a $k$.
Now we proceed similarly to the torsion case.
Pick a  $\xi \in \Spinc(N)$ with $\SW_N(\xi) \neq 0$ and $c_1(\xi) = s$.
By the Baldridge formula in Lemma \ref{BaldrigeFormula} and the choice of $k$,
we have that $\SW_{W_k} (p^*_k \xi) =\SW_N(\xi)$. Thus
the class
$p_k^*s$ is basic as well. The adjunction inequality of Lemma \ref{AdjunctionInequality}
gives
\begin{align*}
\chi_-(\Sigma) \geq \Comp(\sigma_k) &\geq | \sigma_k \cdot \sigma_k | + \langle p_k^*s, \sigma_k\rangle
= | \sigma \cdot \sigma | + \langle s, p_*\sigma \rangle
= | \sigma \cdot \sigma | + \Thur{p_* \sigma}.
\end{align*}
Thus $\Sigma$ is non-degenerate in $W$.
\item[General case] As $N$ has enough basic classes, there is a finite cover $\Map g M N$ and a
basic class $s\in \Bas(M)$ such that $\langle s, g^! p_* \sigma\rangle = \Thur {g^! p_* \sigma}$.
With Lemma \ref{FiniteCovers} we deduce that 
$\langle s, p_* \widetilde \sigma\rangle = \Thur {p_* \widetilde \sigma}$,
where $\widetilde \sigma$ denotes $\widetilde g^! \sigma$. 
The class $\widetilde \sigma$ is the fundamental class of $\widetilde g^{-1}(\Sigma)$. The surface
$\widetilde g^{-1}(\Sigma)$ is disjoint from $\widetilde p^{-1} (g^{-1}(\gamma)) = \widetilde g^{-1} (p^{-1}(\gamma))$.
Furthermore, $g^{-1}(\gamma)$ will be non-torsion as well.
Applying the case above to a component of $g^{-1}(\gamma)$ and to the surface $\widetilde g^{-1}(\Sigma)$,
we obtain the inequality
\begin{align*}
\chi_-(\widetilde g^{-1}(\Sigma)) \geq |\widetilde \sigma \cdot \widetilde \sigma| + \Thur{p_* \widetilde \sigma}.
\end{align*}
Thus $\widetilde g^{-1}(\Sigma)$ is non-degenerate and so is $\Sigma$ by Lemma \ref{FiniteCovers}.
\end{description}
\end{proof}

The alert reader might wonder about the existence of such an embedded loop as required by Lemma
\ref{DisjointLoop}. For a general surface $\Sigma$ such a loop will not exist. 
Nonetheless we describe a way how one can still obtain the 
inequality above. The idea involved here is that by attaching pipes to a surface
we can again ensure the existence of such an embedded loop. This increases the complexity of the surfaces.
But note that in finite covers we still only need one embedded loop whereas the complexity grows linearly.
This discrepancy allows us to obtain non-degenerateness as the limit of inequalities coming from finite covers.
The rest of the paper consists of this argument.
\subsection{The piping construction}
We first describe the piping construction and then finish the proof
of Theorem \ref{FinalEstimate} by considering the non-torsion case.

\begin{lem}\label{Tubing}
Let $\Sigma \subset W$ be an embedded surface
and $\gamma \subset N$ an embedded loop such that $p_* [\Sigma] \cdot_{N} [\gamma] = 0$.
Assume that $\Sigma$ and $p^{-1}(\gamma)$ intersect transversely in $m > 0$ points.
There is an embedded surface $\Sigma' \subset W$ with the properties:
\begin{enumerate}
\item $[\Sigma] = [\Sigma'] \in H_2(W;\IZ)$,
\item the surface $\Sigma'$ does not intersects $p^{-1}(\gamma)$, and
\item $\chi_-(\Sigma') \leq \chi_-(\Sigma) + m$.
\end{enumerate}
\end{lem}
\begin{proof}
We give a construction of a surface $\Sigma'$ which removes two points of intersection.
It will preserve the fundamental class $[\Sigma'] = [\Sigma]$ and 
increase the complexity by at most two: $\chi_-(\Sigma') \leq \chi_-(\Sigma) + 2$. 
This is done by embedding a pipe connecting two intersection points with opposite signs. 

The projected surface $p(\Sigma)$ and $\gamma$ intersect transversely as the pair $\Sigma$ 
and $p^{-1}(\gamma)$ does.
By a small perturbation of $\Sigma$ we assume that no two points of
$\Sigma$ get mapped under $p$ to the same intersection point. 
Pick two points of intersection $x,y$ of opposite signs which lay adjacent on $\gamma$.
Laying adjacent here means that there is  
an arc $c$ connecting $x,y$ in $\gamma$ which does not contain another point of
intersection. 

There is a neighbourhood of $c$ which is a solid cylinder $C$ intersecting
$p(\Sigma)$ in exactly two discs at $x$ and $y$.
As the circle bundle $\Map p W N$ is trivial over $C$ and discs are contractible,
we can lift the cylinder $C$ to a cylinder $P \subset W$ whose ends are discs in $\Sigma$.
This cylinder $P$ will be our pipe. Removing the two discs $P \cap \Sigma$ and adding
the rest of the boundary of $P$, which is an annulus, we obtain a surface $\Sigma'$.
As a result of this description, we obtain the inequality $\chi_-(\Sigma') \leq \chi_-(\Sigma) + 2$.
We have obtained $\Sigma'$ from $\Sigma$ by surgering along an $1$-handle.
As the pipe is embedded we obtain a map from the trace of the surgery to $W$ restricting
to the embeddings of $\Sigma$ and $\Sigma'$ on the two boundaries. 
Thus we get $[\Sigma]= [\Sigma']$.
\end{proof}
Now we can prove the non-torsion case without any restrictions.
By applying the lemma above in finite covers we obtain a family of
inequalities which in the limit give an inequality which 
proves non-degenerateness.
\begin{lem}\label{NonTorsionCase}
	Suppose the Euler class $e$ is non-torsion. Then each
	class $\sigma \in H_2(N;\IZ)$ is non-degenerate.
\end{lem}
\begin{proof}
	The manifold $N$ has enough basic classes. So by definition $N$ has
	a finite cover with $b_1(N) \geq 3$. As it is enough to show that
	the inequality holds in a finite cover, we may assume $N$ itself
	has $b_1(N) \geq 3$.

	Pick an embedded loop $\gamma$ such that $p_* \sigma \cdot \gamma = 0$
	and $[\gamma] \in H_1(N;\IZ)$ non-torsion. This is possible as
	$b_1(N) \geq 3$ and because of the fact that we can represent any element
	of $H_1(N;\IZ)$ by an embedded loop.
	Let $\alpha \in \Hom(\pi_1(N); \IZ)$ be non-zero with 
	$\alpha([\gamma]) = 0$.
	We consider the covers $\Map {g_k} {N_k} {N}$ induced by the kernel
	\begin{align*}
		\pi_1(N) &\rightarrow \ZkZ k\\
		l &\mapsto \alpha(l) 
	\end{align*}
	Note that $g_k^{-1}(\gamma)$ has $k$ components.
	Pick one lift $\gamma_k$ of $\gamma$ out of the these $k$ components.
	Denote the bundle pulled back with $g_k$ by $\Map {p_k} {W_k} {N_k}$. This induces
	a $k$-cover $\Map {\widetilde g_k} {W_k} {W}$.

	Choose an embedded surface $\Sigma \subset W$ of minimal complexity with $[\Sigma] = \sigma$.
	By general position we can assume that $\Sigma$ and $\Gamma := p^{-1} (\gamma)$ 
	have transverse intersection.
	Denote the geometric intersection number by $m$ which is just the cardinality of $\Sigma \cap \Gamma$. 
	Note that $\gamma_k$ is just
	a single lift and not the preimage of $\gamma$. Thus the geometric intersection number 
	between $\Sigma_k := \widetilde g_k^{-1} (\Sigma)$ and $\Gamma_k := p_k^{-1} (\gamma_k)$
	is also just $m$. This follows from the fact that $\widetilde g_k$ restricts to a diffeomorphism
	onto its image in a neighbourhood of $\Gamma_k$. We denote the fundamental class of $\Sigma_k$
	by $\sigma_k$.
	By Lemma \ref{Tubing} there is a surface $\Sigma_k'$ such that
	\begin{enumerate}
		\item $\Sigma_k'$ is disjoint from $\Gamma_k$,
		\item $[\Sigma_k'] = [\Sigma_k]$, and
		\item $\chi_-(\Sigma_k') \leq \chi_-(\Sigma_k) + m$.
	\end{enumerate}
	Now we apply Lemma \ref{DisjointLoop} to $\Sigma_k'$ and $\gamma_k$. We obtain the inequality 
	\begin{align*}
		\chi_-(\Sigma_k) + m \geq \chi_-(\Sigma_k') \geq |\sigma_k \cdot \sigma_k | 
			+ \Thur {{p_k}_* \sigma_k}.
	\end{align*}
	By the multiplicative properties we get
	\begin{align*}
		k \chi_-(\Sigma) +m &\geq k |\sigma \cdot \sigma| + k\Thur {p_* \sigma}.
	\end{align*}
	Thus we have shown the inequality 
	\begin{align*}
		\Comp(\sigma) + \frac{m}{k} &\geq |\sigma \cdot \sigma| + \Thur {p_* \sigma}
	\end{align*}
	for all $k \geq 1$ while keeping $m$ fixed.
	This implies the conclusion.	
\end{proof}

\bibliography{biblo}{}
\bibliographystyle{alpha}
\end{document}